\documentclass[11pt]{amsart}
\usepackage{stmaryrd}
\usepackage{tikz-cd}
\usepackage[utf8]{inputenc}
\usepackage{fullpage, mathrsfs, mathtools, amssymb, tikz,bm, amscd, scrextend,  verbatim, enumerate, amsmath, amsthm}
\usepackage{extarrows}
\newcommand{\eqdef}{\xlongequal{\text{($\ast \ast$)}}}%
\usepackage{wasysym, mathtools, mathrsfs}
\usepackage[colorlinks=true, linkcolor=red!80!black, urlcolor=purple, citecolor=blue!70!black]{hyperref}
\usepackage{microtype}
\usepackage{verbatim}

\usepackage[mathscr]{euscript}

\usetikzlibrary{matrix, arrows}

\input xy
\xyoption{all} 

\usepackage{esvect}

\newtheorem{question}{Question}
\newtheorem{theorem}{Theorem}[section]
\newtheorem{lemma}[theorem]{Lemma}
\newtheorem{prop}[theorem]{Proposition}
\newtheorem{cor}[theorem]{Corollary}

\theoremstyle{definition}
\newtheorem{definition}[theorem]{Definition}
\newtheorem{example}[theorem]{Example}
\newtheorem{remark}[theorem]{Remark}

\theoremstyle{remark} 

\newtheorem{notation}[theorem]{Notation}

\newcounter{step}

\usepackage{faktor}\usepackage{amsmath}\usepackage{amssymb}

\makeatletter
\DeclareRobustCommand*{\mfaktor}[3][]
{
	{ \mathpalette{\mfaktor@impl@}{{#1}{#2}{#3}} }
}
\newcommand*{\mfaktor@impl@}[2]{\mfaktor@impl#1#2}
\newcommand*{\mfaktor@impl}[4]{
	\settoheight{\faktor@zaehlerhoehe}{\ensuremath{#1#2{#3}}}%
	\settoheight{\faktor@nennerhoehe}{\ensuremath{#1#2{#4}}}%
	\raisebox{-0.5\faktor@zaehlerhoehe}{\ensuremath{#1#2{#3}}}%
	\mkern-4mu\diagdown\mkern-5mu%
	\raisebox{0.5\faktor@nennerhoehe}{\ensuremath{#1#2{#4}}}%
}

\newcommand{\Q}{{\mathbb Q}}

\newcommand{\bP}{{\mathbb P}}

\newcommand{\bQ}{{\mathbb Q}}

\newcommand{\calK}{\mathcal K}
\newcommand{\calO}{\mathcal O}
\newcommand{\calM}{\mathcal M}
\newcommand{\calH}{\mathcal H}

\newcommand{\calL}{\mathcal L}

\newcommand{\calU}{\mathcal U}
\newcommand{\calX}{\mathcal X}
\newcommand{\calC}{\mathcal C}

\newcommand{\calY}{\mathcal Y}
\newcommand{\calZ}{\mathcal Z}

\newcommand{\calD}{\mathcal D}
\newcommand{\calV}{\mathcal V}

\newcommand{\sD}{\mathscr{D}}

\newcommand{\sK}{\mathscr{K}}

\newcommand{\sM}{\mathscr{M}}
\newcommand{\sN}{\mathscr{N}}

\newcommand{\sX}{\mathscr{X}}
\newcommand{\sY}{\mathscr{Y}}

\newcommand{\tildeD}{\widetilde{D}}
\newcommand{\tildeX}{\widetilde{X}}

\newcommand{\Supp}{\operatorname{Supp}}
\newcommand{\Diff}{\mathrm{Diff}}
\newcommand{\Spec}{\operatorname{Spec}}
\newcommand{\Proj}{\operatorname{Proj}}

\newcommand{\bfa}{\mathbf{a}}
\newcommand{\bfv}{\mathbf{v}}
\newcommand{\bfb}{\mathbf{b}}
\newcommand{\bfc}{\mathbf{c}}

\author{Kenneth Ascher}
\address{Department of Mathematics, University of California, Irvine, CA, 92697}
\email{kascher@uci.edu}
\author{Dori Bejleri}
\address{Harvard University,
	One Oxford Street,
	Cambridge, MA 02138 USA}
\email{bejleri@math.harvard.edu}
\author{Giovanni Inchiostro}
\address{University of Washington, Department of Mathematics, Box 354350, Seattle, WA 98195 USA}
\email{ginchios@uw.edu}
\author{Zsolt Patakfalvi}
\address{\'Ecole Polytechnique F\'ed\'erale de Lausanne (EPFL), MA C3 635, Station 8, 1015 Lausanne, Switzerland}
\email{zsolt.patakfalvi@epfl.ch}

\title{Wall crossing for moduli of stable log pairs}
\begin{document}
	
	\begin{abstract} 
		We prove, under suitable conditions, that there exist wall-crossing and reduction morphisms for moduli spaces of stable log pairs in all dimensions as one varies the coefficients of the divisor. 
		
	\end{abstract}

	\maketitle

	\section{Introduction}
	
	Compactifying moduli spaces is a central problem of algebraic geometry. It has long been apparent that moduli spaces often admit different compactifications depending on some choice of parameters, and so it is natural to ask how these compactifications and their universal families are related as one varies the parameters. The goal of the present article is to answer this question for compact moduli spaces of higher dimensional \emph{stable log pairs} or \emph{stable pairs} for short. 
	
	A stable pair is a pair $(X, \sum a_i D_i)$ consisting of a variety $X$ and a $\mathbb{Q}$-divisor $\sum a_i D_i$ satisfying certain singularity and stability conditions, which we will recall below. The standard example is a smooth normal crossings pair with $0 < a_i \le 1$ and $K_X + \sum a_i D_i$ ample. Compact moduli spaces of stable pairs with fixed coefficient or \emph{weight vector} $\bfa = (a_1, \ldots, a_n)$ and fixed numerical invariants have been constructed using the tools of the minimal model program (\cite{kollarmodulibook} and Section \ref{section:moduli:of:stable:pairs}). These moduli spaces are quite large and unwieldy in general, and so in practice one studies the closure of a family of interest inside the larger moduli space.  Theorem \ref{theorem:summary} below summarizes our main results in a simplified, but typical situation. We will state our general results in Section \ref{sec:mainthm}. 
	
	\begin{theorem}\label{theorem:summary} Let $(X,D_1, \ldots, D_n) \to B$ be a family of smooth normal crossings pairs over a smooth connected base $B$ and let $P$ be a finite, rational polytope of weight vectors $\bfa = (a_1, \ldots, a_n)$ such that $a_i < 1$ and $(X, \sum a_iD_i) \to B$ is a family of stable pairs for each $\bfa \in P$. Let $\sN_\bfa$ denote the normalized closure of the image of $B$ in the moduli space of $\bfa$-weighted stable pairs with universal family of stable pairs $(\calX_\bfa, \sum a_i \calD_i) \to \sN_{\bfa}$. Then there exists a finite, rational polyhedral wall-and-chamber decomposition of $P$ such that the following hold.
		\begin{enumerate}[(a)]
			\item For $\bfa, \bfa'$ contained in the same chamber, there are canonical isomorphisms
			$$
			\xymatrix{\calX_\bfa \ar[r]^{\cong} \ar[d] & \calX_{\bfa'} \ar[d] \\ 
				\sN_\bfa \ar[r]^{\cong} & \sN_{\bfa'}.
			}
			$$
			
			\item For $\bfa, \bfb \in P$ contained in different chambers and satisfying $b_i \le a_i$ for all $i$, there are canonical birational wall-crossing morphisms
			$$
			\rho_{\bfb, \bfa} : \sN_{\bfa} \to \sN_{\bfb}
			$$
			such that for any third weight vector $\bfc$ with $c_i \le b_i$, we have $\rho_{\bfc, \bfb} \circ \rho_{\bfb, \bfa} = \rho_{\bfc, \bfa}$. Moreover, the map $\rho_{\bfb, \bfa}$ is induced by a birational map $h^{b,a} : \calX_\bfa \dashrightarrow \rho_{\bfb, \bfa}^*\calX_{\bfb}$ such that, for a generic $u \in \sN_\bfa$, the fiberwise map $h^{b,a}_u : (\calX_\bfa)_u \dashrightarrow (\calX_{\bfb})_{\rho_{\bfb, \bfa}(u)}$ is the canonical model of $((\calX_\bfa)_u, \sum b_i (\calD_i)_u)$. 
			
		\end{enumerate}
		
	\end{theorem} 
	
	\begin{remark}
		We note that, to obtain the strongest results,  taking the normalization of the closure in the above theorem is necessary; see Section \ref{sec:normalization} for a discussion and example.
	\end{remark}
	
	Before stating our more general results, let us recap the history and context behind Theorem \ref{theorem:summary}. In dimension one, we have the classical moduli space $\calM_{g,n}$ of smooth projective $n$-pointed curves $(C, p_1, \ldots, p_n)$ of genus $g$ and the Deligne--Mumford--Knudsen compactification $\overline{\calM}_{g,n}$ parametrizing $n$-pointed stable curves of genus $g$. Inspired by ideas from the minimal model program, Hassett in \cite{Hassett} introduced a new family of modular compactifications of $\calM_{g,n}$ depending on a rational weight vector $\bfa = (a_1, \ldots, a_n)$ with $0 < a_i \le 1$ which parametrizes $\bfa$-\emph{weighted stable curves}. 
	
	An $\bfa$-weighted pointed stable curve is a tuple $(C, p_1, \ldots, p_n)$ such that:
	\begin{itemize}
		\item $C$ has genus $g$ and at worst nodal singularities;
		\item the points $p_i$ lie in the smooth locus of $C$ and for any subset $p_{i_1}, \ldots, p_{i_r}$ of points which coincide, we have
		$
		\sum_k a_{i_k} \le 1;
		$
		\item the divisor $K_C + \sum a_i p_i$ is ample.
	\end{itemize}
	When $a_i = 1$ for all $i$, the second condition is the requirement that the $p_i$ are distinct and the third condition is the Deligne--Mumford--Knudsen stability condition, and so we recover $\overline{\calM}_{g,n}$. 
	
	Weighted stable curves form a proper moduli space $\overline{\calM}_{g, \mathbf{a}}$ for $0 < a_i \le 1$ satisfying the condition that $ 2g -2+ \sum_{i=1}^n a_i>0. $
	These conditions define a finite, rational polytope of admissible weight vectors $P$ as in Theorem \ref{theorem:summary}, where the family $(X, D_1, \ldots, D_n) \to B$ is the universal family of smooth $n$-pointed curves of $\calM_{g,n}$. In particular, Hassett  \cite{Hassett} proved Theorem \ref{theorem:summary} in this setting. In fact, in this case $h$ is a birational \emph{morphism} produced as an explicit sequence of contractions of rational tails on which the degree of $K_C + \sum b_i p_i$ is non-positive, that is, $\bfb$-unstable rational tails.

	The natural generalization of a pointed stable curve to higher dimensions, introduced by Koll\'ar and Shepherd-Barron \cite{KSB} and Alexeev \cite{alK^2}, is a stable pair $(X, \sum a_i D_i)$ such that
	
	\begin{enumerate} 
		
		\item $(X, \sum a_i D_i)$ has semi-log canonical singularities (slc, see Definition \ref{def:slc}); and
		\item $K_X + \sum a_i D_i$ is an ample $\mathbb{Q}$-Cartier divisor. 
	\end{enumerate} 
	Explicit stable pair compactifications of moduli of higher dimensional varieties have been studied extensively in recent years, e.g. weighted hyperplane arrangements \cite{hkt,Ale15}, principally polarized abelian varieties \cite{Ale02}, plane curves \cite{Hac04}, and elliptic surfaces \cite{AB17,Inc18}, etc.
	
	Thanks to the combined efforts of many authors (see e.g. \cite[30]{kollarmumford} for a historical summary), there exists a proper moduli space $\sK_{\bfa, v}$ of $\bfa$-weighted stable pairs with volume of $$\mathrm{vol}(K_X + \sum a_i D_i) = v$$
	in all dimensions. For convenience, we often suppress the volume $v$ or consider instead $\sK_{\bfa} := \bigsqcup_v \sK_{\bfa,v}$. Note that the volume will vary as a function of the weight vector $\bfa$ and also changes under wall-crossing morphisms, e.g. in the case of curves, the volume is $2g - 2 + \sum_{i = 1}^n a_i$. 
	
	The basic idea then behind Theorem \ref{theorem:summary} is to consider the universal $\bfa$-weighted stable  family $(\calX_\bfa, \sum a_i \calD_i)$ and run the minimal model program with scaling. This produces the canonical model of $(\calX_\bfa, \sum b_i \calD_i)/\sN_{\bfa}$ and the birational map $h$. We then need to check that this is indeed a stable family of $\bfb$-weighted pairs which then induces the wall-crossing morphism $\rho_{\bfb, \bfa}$. The finite wall-and-chamber decomposition is ultimately a consequence of \cite[Corollary 1.1.5]{BCHM10}.
	
	One complication of the higher dimensional case is that $h$ is in general not a morphism due to the existence of flips. A more serious challenge is that, contrary to the one dimensional case, $\sK_{\bfa, v}$ is in general very singular with many irreducible components parametrizing non-smoothable, reducible varieties \cite{Vakil, Persson}. Moreover, the MMP and even finite generation of the log canonical ring can fail in general. In order to overcome some of the many complications, we need to work with the closure of irreducible loci parametrizing normal crossings, or more generally klt pairs. Indeed, one of the key insights of this paper is that wall-crossing for moduli of stable pairs is controlled by the minimal model program with scaling on the total spaces of $1$-parameter smoothings of the slc pairs on the boundary. Finally, in order to apply the strategy described above, we need to work over some smooth base (e.g. a compactification of $B$ in Theorem \ref{theorem:summary}) and then descend to the seminormalization or normalization of the corresponding moduli space.

	\subsection{Statements of the main results} \label{sec:mainthm}
	
	We are now ready to state our main results in full generality. Fix some weight vector $\bfa = (a_1, \ldots, a_n)$ of rational numbers $a_i \in (0,1] \cap \mathbb{Q}$. Let $f : (X, \sum a_i D_i) \to B$ be a locally stable family (Definition \ref{def:locstab}). 
	
	\begin{definition}\label{def:admissible} We say that a weight vector $\bfb = (b_1, \ldots, b_n)$ is \emph{admissible} if $(X, \sum b_i D_i) \to B$ is locally stable and $K_X + \sum b_i D_i$ is $f$-big. We say that a polytope
		$
		P \subset ((0,1] \cap \mathbb{Q})^n
		$
		is admissible if every vector $\bfb \in P$ is admissible. 
	\end{definition} 
	
	\begin{notation}
		For $\bfb \le \bfa$ admissible weight vectors, we define $\bfv(t) = t\bfa + (1-t)\bfb$ for $t \in [0,1]$.
	\end{notation}
	\begin{notation} For any weight vector $\bfv = (v_1, \ldots, v_n)$ we denote by $\bfv D$ the divisor $\sum v_i D_i$. 
	\end{notation}
	
	Let $\sK^\circ \subset \sK_{\bfa}$ be a quasicompact locally closed substack of the space of $\bfa$-weighted stable pairs, and  suppose that $\sK^\circ$ parametrizes klt pairs. Let $f^\circ : (\calX^\circ, \bfa \calD^\circ) \to \sK^\circ$ denote the universal family of klt stable pairs over $\sK^\circ$. Fix an admissible weight vector $\bfb \le \bfa$ for $f^\circ$. For each $t \in [0,1]$ we have a set theoretic map
	$
	\phi_t : \sK^\circ(k) \to \sK_{\bfv(t)}(k)
	$
	which takes a point $x : \Spec k \to \sK^\circ$ classifying the klt stable pair $(X, \bfa D)$ to the point $x_{\bfv(t)} : \Spec( k) \to \sK_{\bfv(t)}(k)$ classifying the canonical model of $(X, \bfv(t) D)$. 
	
	\begin{definition} For each $t \in [0,1]$ we let $\sM_{t}$ denote the seminormalization of the closure of the image of $\phi_t$ and we let $\sN_t$ denote the normalization of $\sM_t$. We let $\sM_{\bfa}$ (resp. $\sM_{\bfb}$) denote $\sM_0$ (resp. $\sM_1$) and similarly for $\sN$. 
	\end{definition}
	
	\begin{remark} Note that $\sM_t$ and $\sN_t$ are proper Deligne--Mumford stacks with families of $\bfv(t)$-weighted stable pairs pulled back from the universal family of $\sK_{\bfv(t)}$. Moreover, since seminormalization is functorial, the family over $\sM_t$, which we denote $(\calX_{t}, \bfv(t) \calD_t) \to \sM_t$, is the universal family for the functor of stable families $g : (Z, \bfv(t) \Delta) \to B$ over seminormal base schemes $B$ such that for each $b \in B$, the fiber $g_b$ is the limit of a family of canonical models of the pairs parametrized by $f^\circ$. 
	\end{remark}
	
	\begin{theorem}[Theorem \ref{thm:morphismextends}, Corollary \ref{cor:finiteness} and Theorem \ref{theorem:flip-like:morphisms}]\label{thm:mainintro} There exist finitely many rational numbers $t_i \in [0,1] \cap \mathbb{Q}$ with $0 < t_1 < \ldots < t_m < 1$ such that the following hold. 
		
		\begin{enumerate} 
			\item For each $t_i < s < s' < t_{i + 1}$, $\sM_{s} \cong \sM_{s'}$ and the universal families $(\calX_{s}, \bfv(s) \calD_{s})$ and $(\calX_{s'}, \bfv(s') \calD_{s'})$ have isomorphic underlying marked families so that $$(\calX_{s'}, \bfv(s') \calD_{s'}) \cong (\calX_s, \bfv(s') \calD_s).$$ Moreover, these isomorphisms fit in a commutative diagram below.
			\begin{equation*}
				\xymatrix{
					\calX_s \ar[r]^{\cong} \ar[d] & \calX_{s'} \ar[d] \\ \sM_s \ar[r]^{\cong} & \sM_{s'} 
				}
			\end{equation*}
			
			\item For each consecutive pair $t_i < t_{i + 1}$, and any $t_i < s < t_{i + 1}$ there is a commutative diagram 
			$$
			\xymatrix{
				\calX_{t_i} \ar[d] & \ar[l] \ar[r] \ar[d] \calX_{s} & \calX_{t_{i + 1}} \ar[d] \\
				\sM_{t_i} & \sM_{s} \ar[r]^{\beta_{t_{i + 1}}} \ar[l]_{\alpha_{t_i}} & \sM_{t_{i+1}}.
			}
			$$
			where the morphism $\calX_s \to \sM_s$ in the middle is independent of $s$ by part (1). 
			
			\item There is a dense open substack $\calU \subset \sM_s$ parametrizing klt pairs such that for each $u \in \calU$ classifying the klt stable pair $(\calX_u, \bfv(s)\calD_u)$, $\alpha_{t_i}(u)$ classifies the canonical model of $(\calX_u, \bfv(t_i)\calD_u)$ and $\beta_{t_i + 1}(u)$ classifies the canonical model of $(\calX_u, \bfv(t_{t_i + 1})\calD_u)$.
		\end{enumerate}
	\end{theorem} 
	
	In particular, Theorem \ref{thm:mainintro} shows that there are finitely many walls $t_i$ and finitely many moduli spaces parametrizing canonical models of the fibers of $f^\circ$ as we reduce weights from $\bfa$ to $\bfb$ along the line $\bfv(t)$. Moreover, around each wall, the moduli spaces are related via the morphisms $\alpha_{t_i}$ and $\beta_{t_i}$ which we call \emph{flip-like morphisms} as they are induced by flips in the mmp with scaling as one reduces weights from $t_i + \varepsilon$ to $t_i - \varepsilon$. This is a higher dimensional phenomenon not witnessed in the case of curves. 
	
	In order to obtain reduction morphisms as in \cite{Hassett} and in Theorem \ref{theorem:summary}, we need to invert $\beta_{t_i}$. In general, this is only possible up to normalization (see Section \ref{sec:normalization} for an example). 
	
	\begin{theorem}[Theorem \ref{thm:finite} and Theorem \ref{thm:reduction:morphism:birational:for:ZMT}]\label{thm:normintro} The morphism $\beta_{t_i} : \sM_{t_i - \varepsilon} \to \sM_{t_i}$ is quasi-finite, proper, birational and representable. In particular, the induced morphism on normalizations $\beta_{t_i}^\nu : \sN_{t_i - \varepsilon} \to \sN_{t_i}$ is an isomorphism. 
	\end{theorem}
	
	Theorem \ref{thm:normintro} allows us to define reduction morphisms $\rho_{\bfb, \bfa} : \sN_{\bfa} \to \sN_{\bfb}$ by composing the induced maps $\alpha_{t_i}^\nu$ on normalizations with the inverses of $\beta_{t_i}^\nu$ for all $(\bfa \to \bfb)$-walls (see Definition \ref{def:reduction}). Under the assumption that the generic fiber of $f^\circ$ is $\bfv(t)$-weighted stable for all $t \in [0,1]$, which is the case for example in dimension $1$ as well as in the setting of Theorem \ref{theorem:summary}, we have the following. 
	
	\begin{theorem}[Theorem \ref{thm:last:thm} and Corollary \ref{cor:reductionbirational}]\label{thm:introbir} Let $P$ be an admissible polytope of weight vectors such  that the generic fiber of the universal family $(\calX_\bfa, \bfa \calD) \to \sM_\bfa$ is $\bfv$-weighted stable for all $\bfv \in P$. Then, for all $\bfb \le \bfa$ in $P$, the reduction morphisms $\rho_{\bfb, \bfa}: \sN_{\bfa} \to \sN_{\bfb}$ are birational and independent of the choice of path from $\bfa$ to $\bfb$. In particular, $$
		\rho_{\bfc, \bfb} \circ \rho_{\bfb, \bfa} = \rho_{\bfc, \bfa}. 
		$$
	\end{theorem}
	
	In Section \ref{section:countereg}, we give several examples illustrating that Theorem \ref{thm:introbir} is subtle without the extra assumption on the generic fiber of the universal family.

	\subsection{Relations to other work}   The behavior of stable pairs moduli under changing the coefficients has been studied in a few previous cases. In \cite{Ale15}, Alexeev constructed compact moduli spaces of \emph{weighted stable hyperplane arrangements}. These are moduli spaces parametrizing pairs $(X, \sum a_i H_i)$, where $X$ is a
	degeneration of $\bP^n$ and the $H_i$ are the limits of hyperplanes. Among other things, Alexeev shows that there are wall-crossing morphisms as one varies the weights on the $H_i$ as in Theorem \ref{theorem:summary}. This provides alternate compactifications of the spaces of Hacking--Keel--Tevelev \cite{hkt}. Similarly, in \cite{AB17} compact moduli spaces of \emph{weighted stable elliptic surfaces} are constructed (see also \cite{Inc18}). These moduli spaces parametrize pairs of an elliptic surface with the divisor consisting of a section and some weighted (possibly singular) fibers. It is proven that these moduli spaces also satisfy the above wall-crossing morphisms as the weight vector varies. A similar phenomenon has also been recently studied from the viewpoint of K-moduli \cite{ADL}.
	Wall-crossing morphisms play an important role in the study of explicit moduli compactifications, their birational geometry, and for the sake of computations on compact moduli spaces (see e.g. \cite{k3,dp}, the related Hassett--Keel program \cite{fedorchukalternate}, variation of GIT \cite{dhu, thaddeus}, and the Hassett--Keel--Looijenga program \cite{LO, LO2, LO3, ADL2}).

	\subsection*{Conventions}We work over an algebraically closed field $k$ of characteristic 0. All schemes are finite type over $k$, unless otherwise stated. A point will be a closed point, unless otherwise stated.
	Given a morphism $f:\calX \to \calY$ between two separated Deligne--Mumford stacks, the closure of the image of $f$ will be defined as follows. If $X$ (resp. $Y$) is the coarse space of $\calX$ (resp. $\calY$) and $g$ is the morphism $X \to Y$ induced by $f$, then the closure of the image of $f$ will be $\overline{g(X)} \times_Y \calY$. Unless otherwise specified, when we talk about a pair $(X,D)$ we assume that $D>0$ and that $D$ has rational coefficients.
	For $\mathbf{a} = (a_1, \ldots, a_n) \in \bQ^n_{> 0}$ and divisors $D_1, \ldots, D_n$, we will adopt the notation
	$$
	\mathbf{a}D := \sum a_iD_i.
	$$
	If $D$ is a Weil divisor such that each irreducible component of $\Supp(D)$ intersects the smooth locus of $X$, we will make no distinction between $D$ and its associated divisorial subsheaf. 
	
	\begin{remark}
		The theory of $\mathbb{R}$-divisors is not as well-developed from the point of view of moduli theory as compared to the theory of $\mathbb{Q}$-divisors. For example, at the time of writing this article, \cite[Section 11.4]{kollarmodulibook} was not yet available. So for technical reasons, we restrict to $\mathbb{Q}$-divisors. Nevertheless, since by \cite[Corollary 1.1.5]{BCHM10}, the ``walls'' are always rational numbers, we do expect that one can apply the theory in this paper to the setting of $\mathbb{R}$-divisors.
	\end{remark}
	
	\subsection*{Acknowledgements.} We thank Dan Abramovich, Kristin DeVleming, Brendan Hassett, Stefan Kebekus, S\'andor Kov\'acs, Yuchen Liu, Martin Olsson, Roberto Svaldi, Jakub Witaszek, and Chenyang Xu for helpful discussions. We also thank the anonymous referee for their very helpful comments, references, and simplifications of proofs that helped to improve this paper.

	We are especially grateful to J\'anos Koll\'ar for many helpful comments and discussions which greatly helped to improve this paper, particularly for helping simplify the proofs of Proposition \ref{prop:MMP:with:scaling:Jakub} and Theorem \ref{thm:finite}, for sharing Example \ref{ex:JK}, as well as many other clarifications.

	Parts of this paper were completed while authors were in residence at MSRI in Spring
	2019 (NSF No.
	DMS-1440140). K. A. and D. B. were supported in part by NSF Postdoctoral Fellowships. K. A. was partially supported by NSF grant DMS-2140781 (formerly DMS-2001408). G. I. was partially supported by funds from NSF grant DMS-1759514. Zs. P. was partially supported by the following grants: grant \#200021/169639 from the Swiss National Science Foundation,  ERC Starting grant \#804334.

	\section{The moduli space of stable log pairs} \label{section:moduli:of:stable:pairs}
	
	In this section, we recall the definitions and basic setup of the moduli of stable log pairs (or \emph{stable pairs}). We refer the reader to \cite{kollarmodulibook, Kollarsingmmp} for more details on this formalism, and to \cite[Section 2.3]{KM98} for the singularities of the MMP. We begin by recalling the particular kind of singularities appearing on stable log pairs (see \cite[Chapter 5]{Kollarsingmmp}). 
	
	\begin{definition}\label{def:deminormal}A scheme $X$ is \emph{deminormal} if it is $S_2$, and the singularities in codimension one are at worse nodal singularities.
	\end{definition}
	
	Let $\nu : X^{\nu} \to X$ be the normalization of a deminormal scheme. The conductor ideal
	$$
	\mathrm{Ann}(\nu_*\calO_{X^\nu}/\calO_X) \subset \calO_X
	$$
	defines reduced, pure codimension $1$, closed subschemes $\Gamma \subset X$ and $\bar{\Gamma} \subset X^\nu$ collectively referred to as the double locus. 
	
	\begin{definition}\label{def:slc} Let $(X,\Delta)$ be a pair consisting of a deminormal variety $X$ and an effective Weil $\mathbb{Q}$-divisor $\Delta$ whose support does not contain any irreducible component of the double locus. We say $(X,\Delta)$ has \emph{semi-log canonical singularities} (abbreviated slc) if
		\begin{itemize}
			\item $K_X + \Delta$ is $\mathbb{Q}$-Cartier, and
			\item $(X^\nu, \nu^*\Delta + \bar{\Gamma})$ is log canonical. 
		\end{itemize}
		
	\end{definition}
	
	\begin{definition} A \emph{stable log variety} or \emph{stable pair} is a pair $(X, \Delta)$ such that $(X,\Delta)$ has semi-log canonical singularities and $K_X + \Delta$ is ample. 
	\end{definition}
	
	\begin{definition}\label{def:stable:model}
		Given an slc pair $(X,\Delta)$ with $K_X+\Delta$ big and semiample, consider the scheme 
		$$
		Y:=\Proj(\bigoplus_{r \ge 0} H^0(\calO_X(r(K_X+\Delta)))) = \Proj R(X, K_X + \Delta).
		$$ There is a morphism $f:X \to Y$, and we refer to the pair $(Y,f_*\Delta)$ as the \emph{stable model} of $(X,\Delta)$ if $Y$ is $S_2$. When $(X,\Delta)$ has klt singularities, this is the \emph{canonical model} of the pair. 
	\end{definition}
	
	\begin{remark} In our setting, $(X,\Delta)$ is always the central fiber of a degeneration of klt pairs, so the total space of the degeneration is klt. Kawamata--Viehweg vanishing (see Lemma \ref{lemma:bpf:thm}) guarantees that $Y$ is the central fiber of the ample model of the total space, and $Y$ is $S_2$ by a result of Alexeev (see \cite[Theorem 0.1]{AlexeevLimit}). \end{remark}
	
	\subsection{Families of stable pairs}
	Defining families of stable pairs, and especially defining how the divisor $\Delta$ varies, is quite technical. If the base scheme is smooth, then many of the subtleties disappear, and one can give a simple definition of a family of stable pairs (see Definition \ref{def:family:st:pair:over:smooth:base}). The goal of this subsection is to recall the results in \cite[Chapter 4]{kollarmodulibook} that extend the aforementioned definition from smooth bases to reduced bases and to give the general definition of the moduli space of stable pairs in this setting.
	
	
	
	\begin{definition}\label{def:family:st:pair:over:smooth:base}\cite[Corollary 4.55]{kollarmodulibook}
		Let $(X,\Delta)$ be a pair, and let $f:X\to B $ a flat morphism to a smooth scheme $B$. Then $(X,\Delta)\to B$ is a stable family if $(X,\Delta + f^*D)$ is slc for every snc divisor $D\subseteq B$, and $K_{X/B}+\Delta$ is $f$-ample. Note that by the previous condition with $D=0$, the divisor $K_{X/B} + \Delta$ is $\bQ$-Cartier, so this is well-defined.
	\end{definition}
	
	\begin{definition}[{\cite[Definition 4.2 and Theorem 4.3]{kollarmodulibook}}]\label{def:family:of:pairs}
		A \emph{family of pairs} $f:(X,D) \to S$ over a reduced base scheme $S$ is the data of a morphism $f:X \to S$ and an effective Weil $\mathbb{Z}$-divisor 
		$D$ of $X$. This data has to satisfy the following conditions:
		\begin{itemize}
			\item $f:X \to S$ is flat with reduced, connected and $S_2$ fibers of pure dimension $d$;
			\item The nonempty fibers of $\Supp(D) \to S$ are pure dimensional of dimension $d-1$ and every component of $\Supp(D)$ dominates an irreducible component of $S$;
			\item $f$ is smooth at the generic points of $X_s \cap \Supp(D)$, and
			\item For every $s \in S$ we have that $D$ is Cartier in $X$ and flat over $S$ locally around each generic point of $\Supp(D) \cap X_s$ 
		\end{itemize}
	\end{definition}
	\begin{remark}\label{remark:the:terrible:condition:on:family:of:div:is:automatic:over:normal:bases}
		The last point above is automatic when $S$ is normal, given the first three (see \cite[Theorem 4.4]{kollarmodulibook}). Moreover, by our assumptions, we need not distinguish between a Weil divisor and its associated divisorial subscheme (see \cite[Section 4.3]{kollarmodulibook}). More precisely, the closed subscheme associated to $D$ will be the closure of the closed subscheme given by the equation defining $D$ locally around each generic point of $\Supp(D) \cap X_s$.
	\end{remark}
	
	\begin{remark}
		Observe that $D$ in Definition \ref{def:family:of:pairs} is a relative Mumford divisor in the sense of \cite[Definition 4.68]{kollarmodulibook}.
	\end{remark}

	In our case, since there is a relatively big open set $U \subseteq X$ such that $\calO_U(-D|_U) \subseteq \calO_U$ is a relative line bundle, after each base-change $S' \to S$ the pull-back is still a line bundle on the pull-back $U_{S'}$.
	This gives a pull back operation on $U$, and we can extend divisorially to get the pulled back family of divisors on $X_{S'}$.
	This gives a way to pull back a family of $\mathbb{Z}$-divisors, and
	in the case where we instead have a $\mathbb{Q}$-divisor, we can choose an $m$ divisible enough so that $mD$ is a $\mathbb{Z}$-divisor, pull it back as before, and divide the resulting divisor by $m$. This is known as the pull-back with the common denominator definition.
	\begin{notation}
		Given a morphism $g:S' \to S$ and a projective family of pairs $f:(X,D) \to S$, we will denote with $(X_{S'},D_{S'}) \to S'$ the pull-back, defined as above, of $f$ along $g$.
	\end{notation}
	Finally recall that if $f:(X,D) \to S$ and $S' \to S$ are as above, then \[\Supp(D_{S'})=\Supp(h^{-1}(\Supp(D))),\] where $h:X_{S'} \to X$ is the projection (see \cite[Chapter 4]{kollarmodulibook}).
	
	Since in our case it is necessary to label the various components of $D$, we recall the following.
	\begin{definition}[{\cite[Section 4]{kollarmodulibook}}]
		\label{def:stable_family}
		A \emph{family of varieties marked with $n$ divisors} or an \emph{$n$-marked family} over a reduced scheme $S$ is the data of $f:(X;D_1,...,D_n) \to S$ satisfying the following condition: for every $i$, the pair $(X,D_i) \to S$ is a family of pairs, and $X \to S$ is flat with connected and $S_2$-fibers.
	\end{definition}
	Fix $(a_i)_{i=1}^n \in (\mathbb{Q} \cap (0,1])^n$, and consider an $n$-marked family $$f:(X;D_1,...,D_n) \to S$$ such that for every $s \in S$, the pair $(X_s,a_1(D_1)_s+...+a_n(D_n)_s)$ is stable. The functor of such families is not well-behaved. Therefore,
	on needs the following notion of \emph{stable families}:
	
	\begin{definition}
		[{\cite[Definition-Theorem 4.45 and 4.70.3]{kollarmodulibook}}]\label{def:marking}
		A family of varieties marked with divisors $f:(X;D_1,...,D_n) \to B$ over a reduced scheme $B$ is \emph{stable with coefficients in} $a=(a_1,...,a_n)$ if $K_{X/B}+\sum a_iD_i $ is $\mathbb{Q}$-Cartier and the fibers $\left(X_b,\sum a_i(D_i)_b\right)$ are stable pairs. We will often write that $f:(X, \sum a_iD_i) \to B$ is a stable family, or that $f$ is stable.
	\end{definition}
	
	\begin{theorem}[{\cite[Theorems 4.1 and 4.8]{kollarmodulibook}}]\label{thm:thm4.79} Fix a positive rational number $v$, a positive integer $d$, and a vector of positive rational numbers $\mathbf{a}$.
		Then there is a proper Deligne--Mumford stack $\sK_{\mathbf{a},d,v}$ which, for $B$ seminormal, represents the moduli problem of stable families $f:(X,\sum a_i D_i) \to B$ with fibers of dimension $d$ and volume $v$.
	\end{theorem}
	
	\begin{notation}\label{notation:the:moduli:space:we:use} Often, when $d$ plays no role, we will omit the subscript $d$ in $\sK_{\mathbf{a},d,v}$. We denote by $\sK_{\mathbf{a}}:=\bigcup_v \sK_{\mathbf{a},v}$.
	\end{notation}

	Finally, we will need the notion of a locally stable family.
	
	\begin{definition}\cite[Definition-Theorem 4.7]{kollarmodulibook}\label{def:locstab}
		Let $S$ be a reduced scheme and $f: (X, \Delta) \to S$
		a projective family of pairs. Assume that $(X_s,\Delta_s)$ is slc for every $s\in S$. Then $f: (X, \Delta) \to S$ is \emph{locally stable} or
		slc if the following equivalent conditions hold.
		\begin{enumerate}
			\item $K_{X/S} + \Delta$ is $\bQ$-Cartier,
			\item $f_T : (X_T, \Delta_T ) \to T$ is locally stable whenever $T$ is the spectrum of a DVR
			and $q: T \to S$ is a morphism.
		\end{enumerate}
	\end{definition}
	
	\begin{remark} Note that the definition of a family of stable pairs over a reduced base is \'etale local. Therefore, the space $\sK_{\bfa}$ represents the functor of stable families with coefficients $\bfa$ for reduced Deligne--Mumford stacks.
	\end{remark}
	
	\begin{remark} Koll\'ar has introduced a condition on the reflexive powers of relative pluri-canonical sheaves (see \cite[Chapter 9]{kollarmodulibook} and also \cite{AH, BI}) and the K-flatness condition on the family of divisors \cite{Kflat} which give a well-behaved functor of stable families over arbitrary bases representable by a Deligne--Mumford stack locally of finite type whose seminormalization is the space $\sK_{\bfa}$ introduced above. The reason we avoid this and work with seminormalizations in this paper is twofold. First, checking these conditions over non-reduced bases is subtle, and it is not clear that K-flatness in particular is preserved by the constructions in this paper (see especially the proof of Theorem \ref{theorem:flip-like:morphisms}). Second, the reduction morphisms we produce are ultimately only well-defined on the normalization of the moduli space (see Section \ref{sec:normalization}).
	\end{remark}

	\section{Preliminaries from the MMP} 
	
	In this section, we collect some preliminary results from the minimal model program that we need for the proofs of the main theorems. 
	
	\subsection{Dlt modifications and canonical models} Let $(X,D)$ be a log pair with $D$ is a $\mathbb{Q}$-divisor. One of the main obstacles in ``reducing weights'' on the divisor in a stable pair, is that the pair is not necessarily $\Q$-factorial. Indeed, while for a pair $(X,D)$ the divisor $K_X + D$ is required to be $\bQ$-Cartier, there is no reason for $D$ itself to be $\bQ$-Cartier. A somewhat standard approach that allows one to perturb coefficients on a divisor is using dlt modifications.

	\begin{theorem}[Small dlt modification]\cite[Corollary 1.37]{Kollarsingmmp}\label{dltmod}
		Let $(X,D)$ be a dlt pair with $D$ a boundary. There is a proper birational morphism $g: \tildeX \to X$ such that
		\begin{enumerate}
			\item $\tildeX$ is $\bQ$-factorial,
			\item the morphism $g$ is small, 
			\item $\left(\tildeX, g^{-1}_*D\right)$ is dlt, and
			\item $\mathrm{discrep}\left(\tildeX, g^{-1}_*D\right) = \mathrm{discrep}(X, D)$.
		\end{enumerate}
	\end{theorem}

	\begin{definition}\label{def:bigondlt} Let $f : (X,D) \to B$ be a projective morphism such that $(X,D)$ is a dlt pair and let $\Delta$ be any Weil divisor on $X$. We say that $K_X + \Delta$ is $f$-big if its restriction to the generic fiber is big. Note that the generic fiber is normal, so this makes sense.
	\end{definition} 
	
	We will need the following standard lemma and its corollary.
	
	\begin{lemma}\label{lemma:stable:pairs:that:agree:in:codim:1:iso}
		Let $f:(X,D_X) \dashrightarrow (Y,D_Y)$ be a birational rational map of klt pairs that is an isomorphism in codimension one on both $X$ and $Y$, and assume that $f_*(D_X)=D_Y$.
		Assume further that the canonical models of $(X,D_X)$ and $(Y,D_Y)$ exist. Then $f$ induces an isomorphism of canonical models.
	\end{lemma}
	\begin{proof}Let $L_X:=\calO_X(m(K_X+D_X))$ and $L_Y:=\calO_Y(m(K_Y+D_Y))$ with $m$ so that they are both line bundles. Then if $U$ is the open subset where $X$ and $Y$ are isomorphic, $$H^0\left(X,L_X^{\otimes m}\right)=H^0\left(U,\left(L_X^{\otimes m}\right)_U\right)=H^0\left(U,\left(L_Y^{\otimes m}\right)_U\right)=H^0\left(Y,L_Y^{\otimes m}\right)$$
		since the complement of $U$ has codimension at least 2 in both $X$ and $Y$.
		Then the canonical models of $X$ and $Y$ are Proj of the same graded algebra.
	\end{proof}
	\begin{cor}\label{cor:canonical:model:does:not:depend:on:Qfactorialization}
		Let $(X,D)$ be a klt pair, and let $p:X' \to X$ and $q:X'' \to X$ be two small dlt modifications. Then:
		\begin{enumerate}
			\item The pairs $\left(X',p_*^{-1}(D)\right)$ and $\left(X'',q_*^{-1}(D)\right)$  are klt, and
			\item The pairs in $(1)$ have the same canonical model if it exists.
		\end{enumerate}
	\end{cor}
	\begin{proof} (1) follows from \cite[Lemma 2.30]{KM98}. We now show (2). Since $X' \to X$ and $X'' \to X$ are isomorphisms in codimension one, so is $X' \dashrightarrow X''$. The result then follows from Lemma \ref{lemma:stable:pairs:that:agree:in:codim:1:iso}.
	\end{proof}
	\begin{notation}\label{notation:canonical:model}
		Consider a klt pair $\left(X,\sum a_i D_i\right)$ and let $0<b_i\le a_i$. Let $\tildeX \to X$ be a small dlt modification as above, and let $\tildeD_i\subseteq \tildeX$ be the proper transform of $D_i$. We will refer to the canonical model of $\left(\tildeX,\sum b_i \tildeD_i\right)$ as ``the canonical model of $\left(X,\sum b_i D_i\right)$". This is independent of the choice of $\tildeX$ by Corollary \ref{cor:canonical:model:does:not:depend:on:Qfactorialization}. 
	\end{notation}
	
	We also need the following version of the base point free theorem for degenerations of klt pairs.
	\begin{lemma}\label{lemma:bpf:thm}
		Let $R$ be a DVR essentially of finite type over $k$ with closed point $p$. Let $(X,D)$ be a klt pair with a flat proper morphism $f:X \to \Spec(R)$. If $L$ is a nef line bundle such that $L-K_X-D$ is $f$-nef and big, then for $m$ divisible enough, $L_{|X_p}^{\otimes m}$ is base point free and the morphism induced by $|L^{\otimes m}|$ on $X$ restricts to the morphism induced by $\left|L_{|X_p}^{\otimes m}\right|$ on $X_p$.
	\end{lemma}
	In particular, if $(X,D) \to \Spec(R)$ is a stable family (see Section \ref{section:moduli:of:stable:pairs}) such that $(X,D)$ is klt and $K_X+D$ is $f$-nef and big, then $m(K_{X}+D)_{|X_p}=m(K_{X_p}+D_p)$ is semi-ample for $m$ divisible enough. 
	\begin{proof}
		We know from the base point free theorem \cite[Theorem 6-1-13]{KMM}, for $m$ divisible enough, $L^{\otimes m}$ is globally generated and thus its restriction to a fiber is as well.
		To conclude, note that $R^1f_*L = 0$ from relative Kawamata--Viehweg vanishing and thus by cohomology and base change, $H^0(X_p, L_p^{\otimes m}) = H^0(X, L^{\otimes m})|_{X_p}$. 
	\end{proof}
	
	\subsection{MMP with scaling}\label{subsection:mmpscaling} In this subsection, we recall the version of the MMP with scaling we will use throughout the paper.
	We refer the reader to \cite{HK10} and \cite{BCHM10} for more details.
	
	Let $(X,D)$ be a $\mathbb{Q}$-factorial pair with $D$ a big $\mathbb{Q}$-divisor. Assume that $K_X+D$ is big, and let $H$ be an effective divisor such that the pair $(X,D+H)$ is a klt stable pair.
	Then to obtain the stable model of $(X,D)$ one can first run an MMP for $(X,D)$ with scaling by $H$ to obtain a minimal model $(X^{min}, D^{min})$ of $(X, D)$ \cite[Corollary 1.4.2]{BCHM10}.
	After that, one can apply the base point free theorem to the klt pair $(X^{min}, D^{min})$ to get the stable model. 
	
	In our setting, we only assume that $K_X + D$ is big, but not necessarily that $D$ is big. In this case, we may pick a big effective divisor 
	$$
	D' \sim_{\Q} \varepsilon(K_X + D)
	$$
	for $\varepsilon > 0$ small, such that $(X, D + D' + H)$ is klt. Then the canonical model of $(X, D + D')$ is the same as that of $(X,D)$ so we can run MMP with scaling by $H$ on $(X,D + D')$ where now the divisor $D + D'$ is big, and then apply the base point free theorem to compute the canonical model. In particular, we may apply this method to a small dlt modification to compute the canonical models of $(X, D + tH)$ for $t \in [0,1]$ where $H$ is effective and $(X, D + H)$ is a klt stable pair.

	\begin{prop}\label{prop:MMP:with:scaling:Jakub} Let $(X,D) \to B$ be a klt and $\mathbb{Q}$-factorial pair over $B$ with both $D$ or $K_X+D$ big over $B$. Let $H$ be an effective divisor so that the pair $(X,D+H)$ is klt, and let $t_0 \in (0,1)$. Assume that $K_X+D+tH$ is nef
		over $B$ for every $t_0 \le t \le 1$. Let $(Y,D_Y+t_0H_Y)$ be the canonical model of $(X,D+t_0H)$ over $B$. Let $(Z,D_Z+(t_0-\varepsilon)H_Z)$ denote the canonical model of $(Y,D+(t_0-\varepsilon)H)$ over $B$ for $\varepsilon$ small enough (see Notation \ref{notation:canonical:model}).
		
		Then, for all $\varepsilon > 0$ small enough, we have that:
		\begin{enumerate}[(I)]
			\item \label{itm:MMP:with:scaling:Jakub:no_divisors} There is a small birational morphism $Z \to Y$, and
			
			\item   \label{itm:MMP:with:scaling:Jakub:canonical_model_2}    $(Z,D_Z+(t_0-\varepsilon)H_Z)$ is the canonical model of $(X,D+(t_0-\varepsilon)H)$ over $B$.
		\end{enumerate}
	\end{prop}
	
	\begin{proof} Let $g : X \to Y$ denote the natural morphism, and let $$g^+ : (X^+, D^+ + (t_0 - \varepsilon)H^+) \to Y$$ be the canonical model of $(X, D + (t_0 - \varepsilon)H)$ over $Y$ for $0 < \varepsilon \ll 1$. Let $\pi : X \dashrightarrow X^+$ denote the resulting birational contraction. The canonical model is independent of $\varepsilon$ small enough by \cite[Corollary 1.1.5]{BCHM10}. Now $K_X + D + (t_0 + \varepsilon)H$ is nef by assumption, and $K_X + D + t_0H \equiv_g 0$, so $H$ is $g$-nef. On the other hand, $H^+$ is $\mathbb{Q}$-Cartier, $K_{X^+} + D^+ + (t_0 - \varepsilon)H^+$ is $g^+$-ample, and $Y$ is the log canonical model of $(X^+, D^+ + t_0H^+)$ over $B$. It follows that $-H^+ = -\pi_*H$ is $g^+$-ample. We conclude that $g^+$ is small by the following lemma. 
		
		\begin{lemma}\label{lemma:small}
			Let $g: X \to Y$ and $g^+: X^+ \to Y$ be projective and birational. Assume that $\pi: X \dashrightarrow X^+$ is a rational contraction. Let $H$ be an effective, $g$-nef divisor, such that $-\pi_*(H)$ is $g^+$-ample. 
			Then $g^+$ is small.\end{lemma}
		
		\begin{proof} 
			Without loss of generality, we can assume that $\pi$ is a morphism. Indeed, let $h : W \to X$ and $h^+ : W \to X^+$ be a resolution of $\pi$. Then $h^*H$ is effective and nef over $Y$. Since $\pi$ is a rational contraction, $\pi_*H = h^+_*h^*H$, so we may replace $(X,H, \pi)$ with $(W, h^*H, h^+)$. Moreover, we suppose that $g^+$ is not an isomorphism, otherwise we are done.
			
			For the sake of contradiction, suppose that there exists a divisor $E \subset \mathrm{Ex}(g^+)$. Then $\pi$ is an isomorphism over the generic point of $E$. Thus, there exists a curve $C \subset E$ contracted by $g^+$, and a curve $C' \subset X$, such that $C'$ is not contained in the $\pi$-exceptional locus and $\pi_*C' = C$. Note that $F: = \pi^*\pi_*H - H$ is $\pi$-exceptional and $-F$ is $\pi$-nef, so $F$ is effective by the negativity lemma \cite[Lemma 3.39]{KM98}. Since $\pi_*H$ is $g^+$-ample, $H$ is $g = (g^+ \circ \pi)$-nef, and $C' \not \subset \Supp F$, we have that
			$$
			0>C\cdot\pi_*H= \pi_*C'\cdot\pi_*H = C' \cdot \pi^*\pi_*H = C'\cdot (H + F) \ge C'\cdot H > 0,
			$$
			which is a contradiction. \end{proof}

		As $-\varepsilon H^+$ is ample over $Y$ for $0 < \varepsilon \ll 1$, and since $K_Y + D_Y + t_0H_Y$ is ample over $B$, we conclude that \[K_{X^+}+D^+ + (t_0 - \varepsilon)H^+ = (g^+)^*(K_Y + D_Y + t_0H_Y)- \varepsilon H^+\] is ample over $B$. Since the discrepancies of $(X^+, D^+ + (t_0 - \varepsilon)H^+)$ are greater than or equal to those of $(X, D + (t_0 - \varepsilon)H)$, we see that $X^+$ is the canonical model of $(X, D + (t_0 - \varepsilon)H)$ over $B$. Moreover, as $g^+$ is small, it follows that $X^+$ is is also the canonical model of $(Y, D + (t_0 -\varepsilon)H)$ over $B$.  By uniqueness of canonical models, $X^+ = Z$, and the proposition follows.\end{proof} 
	

	\section{Wall-crossing loci in the moduli space}\label{sec:themodulispaces}
	
	The goal of this section is to define the natural moduli spaces $\sM_t$, depending on a parameter $t \in [0,1]$, which admit a wall-crossing structure. The basic idea is as follows. Let $f : (X, \bfa D) \to B$ be a stable family of interest parametrized by some smooth and irreducible base $B$ and denote by $\bfv(t) := t\bfa + (1-t)\bfb$ for $t \in [0,1]$. Suppose furthermore that $K_{X/B} + \bfb D$ is $f$-big. Then taking the relative canonical model of $(X, \bfv(t)D)$ over $B$ gives us an a priori \emph{rational map} $B \dashrightarrow \sK_{\bfv(t)}$. We will see in Theorem \ref{thm:morphismextends} below that, under some mild assumptions, this extends to a \emph{morphism} $\Phi_t : B \to \sK_{\bfv(t)}$ which on some open set is induced by sending $b \in U \subset B$ to the point classifying the canonical model of $(X_b, \bfv(t)D_b)$. 
	
	Then $\sM_t$, defined as the seminormalization of the scheme theoretic image of $\Phi_t$, carries a universal family of $\bfv(t)$-weighted stable pairs which are limits of the canonical models parametrized by $U$. We will see in Corollary \ref{cor:finiteness} that, as $t$ varies, there are only finitely many different moduli spaces $\sM_t$ and finitely many universal families, up to rescaling the boundary.

	\enspace

	\begin{notation}\label{notation:ab} For coefficient $n$-vectors $\bfa, \bfb$, we write $\bfb \le \bfa$ if $b_i \le a_i$ for all $i = 1, \ldots, n$. For $t \in [0,1]$, we will denote $\bfv(t):= t\bfa + (1-t)\bfb$. 
	\end{notation}
	
	We are now ready to present the main theorem of this section.
	
	\begin{theorem}\label{thm:morphismextends}
		Let $f: (X, \mathbf{a}D) \to B$ be a stable family over a smooth irreducible quasi-projective scheme $B$. Suppose that the generic fiber is klt and that $K_X + \mathbf{v}(t)D$ is $f$-big for each $t \in [0,1]$.
		
		
		\begin{enumerate}\setlength\itemsep{.5em}
			\item  There exists a unique morphism $\Phi_t:B \to \sK_{\bfv(t)}$ and a nonempty open subset $U \subset B$ such that $\Phi_t(u)$ is the point classifying the canonical model of $(X_u,\mathbf{v}(t)D_u)$ for all $u \in U$;
			
			\item There are finitely many $t_i \in \mathbb{Q}$, with $0=t_0 <t_1<...<t_m=1$, which satisfy the following condition. If we denote by $(Z_{t},\mathbf{v}(t)\Delta_{t})$ the family of stable pairs classified by $\Phi_t$, then for every
			$t_i < s_1 \le s_2 <t_{i+1}$ the underlying $n$-marked families $(Z_{s_1}; \Delta_{s_1,1}, \ldots, \Delta_{s_1, n})$ and $(Z_{s_2}; \Delta_{s_2,1}, \ldots, \Delta_{s_2, n})$ are equal, so that $(Z_{s_2},\mathbf{v}(s_2)\Delta_{s_2})=(Z_{s_1},\mathbf{v}(s_2)\Delta_{s_1})$.
			\item For every $t \in [0,1]$, the stable family $f_t : (Z_t,\mathbf{v}(t)\Delta_t) \to B$ is the relative canonical model of $(X,\mathbf{v}(t)D)$ over $B$.
		\end{enumerate}
	\end{theorem}
	
	\begin{remark}
		Observe that, in the particular case where the divisor $K_X+\mathbf{v}(t)D$ restricted to the generic fiber is \emph{ample} for every $t \in [0,1]$, we automatically have a non-empty open subset $U$ and a morphism $U \to \sK_{\mathbf{v}(t)}$. In this special case, the content of the theorem is that we can extend this morphism to $B$. This is the case, for example, in dimension one \cite{Hassett}. 
	\end{remark}
	The proof proceeds as follows. We first show the existence of the rational numbers $t_i$, the so-called walls. We will begin by defining $f_t:(Z_t,\mathbf{v}(t)\Delta_t) \to B$ as the canonical model of $(X,\mathbf{v}(t)D)$ over $B$. Since $B$ is smooth, \cite[Theorem 4.54]{kollarmodulibook} guarantees that 
	$f_t$ is stable, whereas \cite{BCHM10} provides us with the finitely many $t_i$. Finally, to show that $f_t$ is the relative canonical model over an open set of the base, we use an invariance of plurigenera result of \cite[Section 4]{hacon2013birational}.
	
	
	\begin{proof}[Proof of Theorem \ref{thm:morphismextends}]
		We begin by observing that, since the generic fiber of $f$ is klt, the pair $(X,\mathbf{a}D)$ is klt from \cite[Corollary 4.56]{kollarmodulibook}. If $X$ was also $\mathbb{Q}$-factorial, we would consider
		the canonical model $(Z_t,\mathbf{v}(t)\Delta_t)$ of the pair $(X,\mathbf{v}(t)D)$ over $B$. The morphism $(Z_t,\mathbf{v}(t)\Delta_t)\to B$ would be stable (\cite[Corollary 4.57]{kollarmodulibook}), and would induce the morphisms $\Phi_t$.
		However, since $X$ may not be $\mathbb{Q}$-factorial, we need to replace $X$ with a small $\mathbb{Q}$-factorial modification in the argument above. 
		In particular, consider a small $\mathbb{Q}$-factorial modification $\pi: \tildeX \to X$, let $\mathbf{a}\tildeD$ be the proper transform of $\mathbf{a}D$, and denote by $\tilde{f} : \tildeX \to B$ the composition $f \circ \pi$. Since $\pi$ is small, observe that
		\begin{itemize}
			\item $\pi^*(K_X+\mathbf{a}D)= K_{\tildeX}+\mathbf{a}\tildeD$, so $K_{\tildeX} + \mathbf{a}\tildeD$ is $\tilde{f}$-big and $\tilde{f}$-nef over $B$ since it is the pull-back of an $f$-ample divisor;
			\item $\pi_*(\bfa D) = \bfa D$ and the discrepancies of $(\tildeX,\mathbf{a}\tildeD)$ are the same as those of $(X,\mathbf{a}D)$, and
			\item $\pi^{-1}_*(\mathbf{a}D) = \mathbf{a}\tildeD$.
		\end{itemize}
		In particular, the pair $(\tildeX,\mathbf{a}\tildeD)$ is a weak canonical model of $(X,\mathbf{a}D)$, and from \cite[Corollary 4.57]{kollarmodulibook} the morphism $(\tildeX,\mathbf{a}\tildeD)\to B$ is locally stable. Now $\tildeX$ is $\mathbb{Q}$-factorial, so for every $t \in [0,1]$ the morphism $(\tildeX,\mathbf{v}(t)\tildeD)\to B$ is also locally stable. Then we can run MMP with scaling by $(\bfa - \bfb)\tilde{D}$ as described in Subsection \ref{subsection:mmpscaling} to take the canonical model $(Z_t,\mathbf{v}(t)\Delta_t)$ of the pair $(X,\mathbf{v}(t)D)$ over $B$ for all $t \in [0,1]$. By \cite[Corollary 4.57]{kollarmodulibook} the map $(Z_t,\mathbf{v}(t)\Delta_t)\to B$ is stable.

		Now the key input is \cite[Corollary 1.1.5]{BCHM10}. Indeed, by \emph{loc. cit.} there are rational numbers $t_i$ with $0=t_0 < t_1 < ... <t_m=1$
		such that, for every $t_i < s_1 \le s_2 < t_{i+1}$, the pair $(Z_{s_1},\mathbf{v}(s_1) \Delta_{s_1})$ is
		obtained from $(Z_{s_2},\mathbf{v}(s_2) \Delta_{s_2})$ by perturbing the coefficients, i.e. the underlying marked varieties are the same so that $$(Z_{s_2},\mathbf{v}(s_2) \Delta_{s_2})=(Z_{s_1},\mathbf{v}(s_2) \Delta_{s_1}).$$ 
		
		We are left with proving that there exists an open subset $U \subset B$ such that the morphisms $\Phi_t$ on $U$ can be described by sending a pair $(X_u, \mathbf{a}D_u)$ to the canonical model of $(X_u, \mathbf{v}(t)D_u)$. Or in other words, that taking the canonical model of $(\tildeX,\bfv(t)\tildeD)$ gives the fiberwise canonical models.
		
		We begin by fixing a $t$ and taking a log-resolution $\xi:(Y,\Delta_Y)\to (\tildeX,\bfv(t)\tildeD)$, where we denote by $\Delta_Y$ the divisor on $Y$ such that $K_Y+\Delta_Y = \xi^*(K_{\tildeX}+\bfv(t)\tildeD)$. From \cite[Proposition 2.36]{KM98}, we may assume $\Delta_Y$ is of the form $\Delta_Y=\Delta_Y^+ - \Delta_Y^-$, where $\Delta_Y^+$ and $\Delta_Y^-$ are effective, $\Delta_Y^-$ is $\xi$-exceptional, and $\Delta_Y^+$ is smooth. Then from \cite[Corollary 3.53]{KM98}, the canonical models of $(Y,\Delta_Y^+)$ and $(\tildeX,\bfv(t)\tildeD)$ agree. Moreover, we can find an open subset $U\subseteq B$ where the morphisms $Y|_U\to U$ and $\Supp(\Delta_Y^+)|_U\to U$ are smooth, as being smooth is an open condition. Therefore, we can now apply \cite[Theorem 4.2]{hacon2013birational}: the formation of the canonical models commutes with base change. So for every $u\in U$, the canonical model of $(Y,\Delta_Y^+)_u$ is the fiber over $u$ of the canonical model of $(Y,\Delta_Y^+)\to U$, namely $(Z_t,\bfv (t)\Delta_t)_u$. To conclude, note that after further shrinking $U$, we can assume that $(Y, \Delta_Y) \to (\tilde{X}, \bfv(t)\tilde{D})$ is a fiberwise log resolution for $u \in U$. Then
		$$
		K_{Y_u} + (\Delta^+_Y)_u - \xi_u^*(K_{\tilde{X}_u} + (\bfv(t)\tilde{D})_u) 
		$$
		is both effective and $\xi_u$-exceptional, so the log canonical model of $(\tilde{X}_u, (\bfv(t)\tilde{D})_u)$ equals the log canonical model of $(Y_u, (\Delta_Y^+)_u)$, which equals $(Z_t, \bfv(t)\Delta_t)_u$ as required.\end{proof}

	\begin{remark}\label{remark:unique:walls}
		As phrased, the set of rational numbers $\{t_i\}_{i = 0}^m$ of Theorem \ref{thm:morphismextends} is not unique, as we can always subdivide the interval $[0,1]$ further by adding extra $t_j$ and relabeling. However, there is a minimal choice for this set, given by the intersection of all the possible sets of $t_i$. These are the $t_i$ where the canonical models $(Z_t, \bfv(t)\Delta_t)$ actually change. 
	\end{remark}
	
	This leads to the following definition.
	
	\begin{definition}\label{definiton:walls}
		Given  $\mathbf{b} \le \mathbf{a} $ and $B$ as above, we will denote by $(\mathbf{a}\to\mathbf{b})$-\emph{walls}, pronounced ``$\bfa$-to-$\bfb$ walls'', the minimal choice of numbers $$0 = t_0 <  \ldots < t_i < \ldots < t_m = 1$$ as in Theorem \ref{thm:morphismextends}. 
	\end{definition}

	\begin{remark}\label{remark:consequences:second:bullett:thm:4.4}
		We record two consequences of Theorem \ref{thm:morphismextends}:
		\begin{itemize}
			\item For every $i$ and for every rational $s \in (t_i, t_{i+1})$, the divisor $ \mathbf{v}(s) \Delta_{s} $ is $\bQ$-Cartier on $Z_{s}$, and
			\item The pair $\left(Z_{t_i}, \mathbf{v}(t_i) \Delta_{t_i} \right)$ is the canonical model of $\left(Z_{s_{1}}, \mathbf{v}(t_i) \Delta_{s_{1}} \right)$, and $\left(Z_{t_{i+1}}, \mathbf{v}(t_{i+1}) \Delta_{t_{i+1}} \right)$ is the canonical model of $\left(Z_{s_{1}}, \mathbf{v}(t_{i+1}) \Delta_{s_{1}} \right)$ for $t_i < s_i < t_{i + 1}$. \end{itemize}

		The first consequence holds since, for every $t_i<s_1<s_2<t_{i+1}$,
		the  divisors $K_{Z_{s_2}}+\mathbf{v}(s_2)\Delta_2$ and $K_{Z_{s_2}}+\mathbf{v}(s_1)\Delta_2$ are $\mathbb{Q}$-Cartier, so their difference is also $\mathbb{Q}$-Cartier.
		The second consequence follows from the definition of the canonical model. In particular, to check \cite[Definition 3.50 (4)]{KM98}, one can use that the discrepancies of a pair $(X,\sum a_i D_i)$ are continuous functions of the coefficients $a_i$.
	\end{remark}

	We are ready to define the moduli spaces $\sM_t$ which form the natural setting for wall-crossing. 
	
	\begin{definition}\label{Def:calM} Let $f: (X, \bfa D) \to B$ be a stable family satisfying the conditions of Theorem \ref{thm:morphismextends} and suppose that $B$ is proper. Let $\Phi_t$ be as in the conclusion of the theorem. Define $\sM_t$ to be the seminormalization of the image of $\Phi_t:B \to \sK_{\mathbf{v}(t)}$ for $t \in [0,1]$. We will denote by $(\calX_t, \bfv(t) \calD_t)$ the universal family of $\bfv(t)$-weighted stable pairs over $\sM_t$. We will denote by $\sM_\bfa$ and $(\calX_\bfa, \bfa \calD_\bfa)$ (resp. $\sM_\bfb$ and $(\calX_\bfb, \bfb\calD_\bfb)$) the case when $t = 1$ (resp. $t = 0$). 
	\end{definition}
	
	\begin{remark}\label{rem:Mt:proper} Note that $\sM_t$ is proper as both $B$ and $\Phi_t$ are proper, and the seminormalization preserves properness. 
	\end{remark}
	
	\begin{remark}\label{rem:situations}
		
		The reader should keep in mind the following situations which are the most common in practice, noting that the setup of Theorem \ref{thm:morphismextends} allows us the flexibility to consider more general settings. 
		
		\begin{itemize} 
			\item Given a stable family of snc pairs of interest $(X^0, \bfa D^0) \to U$ over a smooth but non-proper base (e.g. $(\mathbb{P}^n, \text{smooth hypersurface})$) we have an induced map $U \to \sK_\bfa$. This may be compactified to a map $B \to \sK_\bfa$ from a smooth, proper base $B$ using \cite[Théorème 16.6]{LMB00}, Chow's Lemma and resolution of singularities. Pulling back the universal family to $B$ gives us a family $(X, \bfa D)$ of stable pairs for which we can apply the proposition. In this case, $\sM_t$ can be thought of as the seminormalization of the $\bfv(t)$-weighted stable pair compactification of the original family of interest. 
			
			\item Let $\sK_0 \subset \sK_{\bfa}$ be some irreducible component of the moduli space $\sK_{\bfa}$ which generically parametrizes klt pairs. Then as above, up to taking a finite cover by a scheme and resolving singularities, we obtain an $\bfa$-weighted stable family $f : (X, \bfa D) \to B$ over a smooth and proper base with a morphism $B \to \sK_\bfa$ dominating the component $\sK_0$. In this case, $\sM_1$ is simply the seminormalization of $\sK_0$. If we assume further that a generic pair lying over $\sK_0$ is $\bfv(t)$-weighted stable for all $t \in [0,1]$, then $\sM_t$ are birational models of $\sM_1$ which carry $\bfv(t)$-weighted stable families. 
			
			\item Let $\sK^\circ \subset \sK_{\bfa}$ be a reduced and irreducible locally closed substack which parametrizes klt pairs. After shrinking $\sK^\circ$, we can assume without loss of generality that it is smooth. After taking a finite cover of the closure of $\sK^\circ$ and resolving singularities, we obtain a stable family $f : (X, \bfa D) \to B$ such that $B$ dominates $\sK^\circ$ under the morphism $\Phi_1 : B \to \sK_{\bfa}$. Then $\sM_1$ is the seminormalization of the closure of $\sK^\circ$, and under the assumptions of Theorem \ref{thm:morphismextends}, the $\bfa$-weighted stable klt pairs parametrized by $\sK^\circ$ are also klt and stable with weights $\bfv(t)$ all $t \in [0,1]$ and thus $\sK^\circ$ admits a monomorphism to $\sK_{\bfv(t)}$, which extend to the morphisms $\Phi_t : B \to \sK_{\bfv(t)}$ given by the theorem. Thus $\sM_t$ are birational models of $\sM_1$ carrying $\bfv(t)$-weighted stable families as before. This case is a hybrid of the above two. 
			
		\end{itemize} 
	\end{remark}

	We have the following immediate corollary of Theorem \ref{thm:morphismextends}. 
	
	\begin{cor}\label{cor:finiteness} For each $t_i < s < s' < t_{i + 1}$,  $\sM_{s} \cong \sM_{s'}$ and the universal families $(\calX_{s}, \bfv(s) \calD_{s})$ and $(\calX_{s'}, \bfv(s') \calD_{s'})$ have isomorphic underlying marked families so that $(\calX_{s'}, \bfv(s') \calD_{s'}) \cong (\calX_s, \bfv(s') \calD_s)$. Moreover, these isomorphisms fit in a commutative diagram below, where each side is cartesian.
		\begin{equation*}
			\xymatrix{
				Z_s \ar@{=}[r] \ar[d] \ar[drr] & Z_{s'} \ar[d] \ar[drr] & \\
				B \ar@{=}[r] \ar[drr] & B \ar[drr] & \calX_s  \ar[r]^(.4){\cong} \ar[d] & \calX_{s'} \ar[d]  \\
				& & \sM_s \ar[r]^(.4){\cong} & \sM_{s'} 
			}
		\end{equation*}
	\end{cor} 
	
	\begin{proof} We claim that the morphism $(\calX_s,\mathbf{v}(s')\calD_s) \to \sM_{s'}$ is locally stable. Since $(\calX_s,\mathbf{v}(s')\calD_s) \to \sM_{s}$ is a well-defined family of pairs, to prove the claim we can use Definition \ref{def:locstab}. In particular, by \cite[Definition-Theorem 4.7]{kollarmodulibook} it suffices to check that for every DVR $R$ and for every morphism
		$T=\Spec(R) \to \sM_{s}$, the family $(\calX_s,\mathbf{v}(s')\calD_s)_T \to T$ is locally stable. 
		By Riemann-Hurwitz, it suffices to check that the family is locally stable, after a further possibly ramified extension of DVRs $T' \to T$, as in the proof of \cite[Proposition 2.10]{kollarmodulibook} (see also \cite[Section 11.23]{kollarmodulibook}). In particular,
		by the valuative criterion of properness, we can assume that the morphism $T \to \sM_{s}$ factors through $\Phi_{s}:B \to \sM_{s}$ as follows.
		$$\xymatrix{ B \ar[rr]^{\Phi_{s}}& & \sM_{s} \\
			& T. \ar[ur] \ar[ul] & }$$
		Thus, we can replace $\sM_{s}$ and its universal family with $B$ and the family lying over $B$.
		The claim then follows from Theorem \ref{thm:morphismextends} (2).
		
		Observe now that $(\calX_s,\mathbf{v}(s')\calD_s) \to \sM_{s}$ is in fact stable, i.e.  $K_{\calX_s/\sM_{s}}+\mathbf{v}(s')\calD_s$ is relatively ample over $\sM_{s}$. Indeed, by Theorem \ref{thm:morphismextends} (2) it is relatively ample when pulled back to $B$, and $B \to \sM_{s}$ is a proper surjection. Therefore, the family $(\calX_s,\mathbf{v}(s')\calD_s) \to \sM_{s}$ induces a morphism $\sM_{s} \to \sM_{s'}$. The argument is symmetric in $s$ and $s'$ so we also have a morphism in the other direction. 
		
		Finally, the fact that these morphisms are inverses and are induced by isomorphisms $$(\calX_s', \bfv(s')\calD_{s'}) \cong (\calX_s, \bfv(s')\calD_s)$$ can be checked pointwise over the moduli space and fiberwise on the universal family and thus follows from Theorem \ref{thm:morphismextends} (2). Commutativity is clear by construction.\end{proof}
	
	Given the corollary, we will introduce the following notation. 
	
	\begin{notation}\label{notation:chamber} For consecutive walls $t_i < t_{i+1}$, we will denote by $$(\calX_{(t_i, t_{i + 1})}, \calD_{(t_i, t_{i + 1})}) \to \sM_{(t_i,t_{i + 1})}$$ the moduli space and universal family of varieties marked with divisor for any $s \in (t_i, t_{i+1})$. 
	\end{notation}
	
	\section{Flip-like morphisms}\label{section flip-like morphisms}
	
	In this section, we will prove the existence of \emph{flip-like morphisms} that relate the moduli spaces $\sM_t$ defined in the previous section as $t$-varies across the $(\bfa \to \bfb)$-walls. With notation as in \ref{notation:ab}, suppose we are in the situation of Theorem \ref{thm:morphismextends}. Recall that the spaces $\sM_t$ as in Definition \ref{Def:calM} admit morphisms $B \to \sM_t \to \calK_{\bfv(t)}$. If $0 = t_0 < \ldots < t_i < \ldots < t_m = 1$ are the $(\bfa \to \bfb)$-walls and $t_i < s_i < s_i' < t_{i + 1}$, then the flip-like morphisms assemble into the diagram below. 
	\begin{equation}\label{fliplike:diagram}
		\xymatrix{
			B \ar[d] \ar[dr] \ar[drr] \ar[drrr] \ar[drrrr]  & \\
			\sM_{\mathbf{a}}=\sM_{t_0} \ar[d] & \sM_{s_0} \ar[l] \ar@{<->}[r]_{\cong} \ar[d] & \sM_{s_0'} \ar[r] \ar[d] &  \ar[d]   \sM_{t_1} & \sM_{s_1}\ar[l] \ar[d] \ar@{<->}[r]_{\cong} & \dots  & \\
			\sK_{t_0} &  \sK_{s_{0}} &  \sK_{s_{0}'} &  \sK_{t_1} &  \sK_{s_{1}} & \dots &  
		} 
	\end{equation}
	
	By Theorem \ref{thm:morphismextends}, we obtain a diagram without the horizontal arrows where the composition $B \to \sK_{t}$ is the morphism $\Phi_t$ and $\sM_t$ is the seminormalization of the image of $\Phi_t$, and by Corollary \ref{cor:finiteness}, we have the horizontal isomorphisms $\sM_{s_0} \cong \sM_{s_0'}$. We can summarize the situation as follows.  
	
	\begin{equation}
		\xymatrix{
			B \ar[d] \ar[dr] \ar[drr] \ar[drrr] \ar[drrrr] \\
			\sM_{t_0} & \sM_{(t_0,t_1)}  &  \sM_{t_1} & \sM_{(t_1,t_2)}  & \sM_{t_2} \dots  
		}    
	\end{equation}
	
	Note here that $\sM_{(t_0, t_1)}$ admit morphisms to $\sK_{s}$ for each $t_0 < s < t_1$, but the target and these morphisms are actually varying even though the source moduli space is independent of $s$. 
	
	For each $t$, the base $B$ carries a stable family $(Z_t, \bfv(t)\Delta_t)$ which is pulled back from the universal family $(\calX_t, \bfv(t)\calD_t) \to \sM_t$. We know from Theorem \ref{thm:morphismextends} that the marked pair $(Z_s, \Delta_s)$ is independent of $s$ for $t_i < s < t_{i + 1}$ with only the coefficients changing. Moreover, $(Z_{t_i}, \bfv(t_i)\Delta_{t_i})$ and $(Z_{t_{i + 1}}, \bfv(t_{i+1})\Delta_{t_{i + 1}})$ respectively are the canonical models of $(Z_s, \bfv(t_i)\Delta_s)$ and $(Z_s, \bfv(t_{i + 1})\Delta_s)$ (Remark \ref{remark:consequences:second:bullett:thm:4.4}). We showed in Corollary \ref{cor:finiteness} that the first fact descends to a statement on the universal family. Putting this together, we have the diagram below, where the squares coming out of the paper are cartesian. 
	
	\begin{equation}
		\label{eq:diag_to_be_filled}
		\resizebox{\linewidth}{!}{%
			\xymatrix{
				(Z_{t_i}, \mathbf{v}(t_i) \Delta_{t_i} )  \ar[d] \ar[ddr]  & \ar[l] (Z_{s}, \mathbf{v}(s) \Delta_{s} )  \ar[d] \ar[r] \ar[ddr]
				&
				(Z_{t_{i+1}}, \mathbf{v}(t_{i+1}) \Delta_{t_{i+1}} ) \ar[d] \ar[ddr]
				\\
				B \ar[ddr] \ar@{=}[r] & \ar[ddr] B \ar@{=}[r] & \ar[ddr] B  \\
				& (\calX_{t_i}, \mathbf{v}(t_i)\calD_{t_i} ) \ar[d] & 
				(\calX_{(t_1, t_{i +1})}, \mathbf{v}(s)\calD_{(t_1, t_{i+1})} )  \ar[d] &  
				(\calX_{t_{i+1}}, \mathbf{v}(t_{i+1})\calD_{t_{i+1}} )  \ar[d] &  
				\\
				& \sM_{t_i}  & \sM_{(t_1, t_{i + 1})} & \sM_{t_{i+1}}
		}}
	\end{equation}

	Note that the \emph{a priori} rational maps $Z_s \to Z_{t_i}$ given by taking the canonical model are actually morphisms. Indeed, since $(Z_{s}, \bfv(t_i)\Delta_{s})$ and $(Z_{s_i}, \bfv(t_{i+1})\Delta_{s})$ are good minimal models, so they admit morphisms to their canonical models. The idea now is to descend these morphisms to the universal families and use them to induce the flip-like morphisms $\sM_{t_i} \leftarrow \sM_{(t_i, t_{i + 1})} \rightarrow \sM_{t_{i + 1}}$. 
	
	\begin{theorem}[Flip-like morphisms]\label{theorem:flip-like:morphisms}In the setting of Theorem \ref{thm:morphismextends}, consider $t_i < t_{i + 1}$ consecutive $(\bfa \to \bfb)$-walls. There is a commutative diagram 
		\begin{equation}
			\xymatrix{
				\calX_{t_i} \ar[d] & \ar[l] \ar[r] \ar[d] \calX_{(t_i, t_{i+1})} & \calX_{t_{i + 1}} \ar[d] \\
				\sM_{t_i} & \sM_{(t_i, t_{i + 1})} \ar[r]^{\beta_{t_{i + 1}}} \ar[l]_{\alpha_{t_i}} & \sM_{t_{i+1}} 
			}
		\end{equation} 
		which commutes with diagram \ref{eq:diag_to_be_filled}. Moreover, we have
		\begin{enumerate}[(i)]
			\item $\alpha_{t_i}$ is induced by taking a pair $(X, \bfv(t_i)\Delta)$ to $\operatorname{Proj}R(K_X + \bfv(t_i)\Delta)$ with the pushfoward divisor,
			\item $\beta_{t_{i+1}}$ is induced by taking a pair $(X, \bfv(t_{i + 1})\Delta)$ to $\operatorname{Proj} R(K_X + \bfv(t_{i+1})\Delta)$ with the pushforward divisor. 
		\end{enumerate}
		
		In particular, over the dense open subset $\calU \subset \sM_{(t_i,t_{i+1})}$ parametrizing klt pairs, $\alpha_{t_i}$ and $\beta_{t_{i+1}}$ can be described by taking fiberwise canonical models.

	\end{theorem}
	
	\begin{proof} We will denote $(\calX_{(t_i, t_{i+1})}, \calD_{(t_i, t_{i + 1})}) \to \sM_{(t_i, t_{i+1})}$ by $(\calX, \calD) \to \sM$ for convenience. 
		
		We need to construct a well-defined family  of pairs over $\sM$ but with coefficients $\mathbf{v}(t_i)$, that pulls back to $(Z_{t_i}, \mathbf{v}(t_i) \Delta_{t_i}) $ over $B$ and similarly for coefficients $\bfv(t_{i + 1})$.
		In particular, we have to show that the canonical model map 
		\begin{equation}\label{eqn:lc}
			(Z_s, \bfv(t)\Delta_s) \dashrightarrow (Z_t, \bfv(t)\Delta_t)
		\end{equation}
		is a morphism which is pulled back from a morphism of families of $\sM$ for any $s \in (t_i, t_{i + 1})$ and $t = t_i, t_{i + 1}$. Note that (\ref{eqn:lc}) is in fact a morphism for $t = t_i, t_{i + 1}$ by the basepoint free theorem applied to $K_{Z_s} + \bfv(t)\Delta_t$. By construction, $(Z_s, \Delta_s)$ is the pull-back of the universal pair $(\calX, \calD) \to \sM$. 
		
		Our task is to descend the stable family $(Z_t, \bfv(t)\Delta_t) \to B$ to a stable family $(\calY, \bfv(t)\calD_\calY) \to \sM$ along with a log canonical linear series $\calX \dashrightarrow \calY$ whose construction is compatible with basechange.  
		
		By \ref{cor:finiteness}, we know that for every $s \in (t_i, t_{i + 1})$ the family $(\calX,\mathbf{v}(s)\calD) \to \sM$ is stable.
		Therefore, since both having lc singularities and being nef are  closed conditions on the coefficients of the divisor, the morphism $\pi:(\calX,\mathbf{v}(t)\calD) \to \sM$ is locally stable, and  $K_{\calX/\sM}+\mathbf{v}(t)\calD$ is $\pi$-nef. We define, for $d$ divisible enough, 
		$$
		\calY:=\Proj_\sM\left(\bigoplus_{m \in \mathbb{N}}(\pi_*\calO_{\calX}(md(K_{\calX/\sM}+\mathbf{v}(t)\calD))\right).
		$$
		That is, $\calY$ is the relative canonical model of $(\calX, \bfv(t)\Delta)$. 
		
		We claim that the construction of $\calX \dashrightarrow \calY$ commutes with base change. By cohomology and base change it suffices to prove that for $d$ and $m$ divisible enough, and for every $p \in \sM$, we have
		$$H^1\left(\calX_p,md\left(K_{\calX_p}+\mathbf{v}(t)\calD_p\right)\right)=0.$$
		Recall that from the definition of $\sM$, every pair appearing as a fiber of $\pi$ can be obtained as the degeneration of a klt pair over a DVR and moreover that $K_{\calX_p} + \bfv(t)\calD_p$ is big and nef. Now the desired vanishing follows from relative Kawamata--Viehweg vanishing as in \cite[Theorem 8.1]{Inc18}. 
		
		In particular, for every $p \in \sM$, 
		\begin{equation*}
			\calY_p = \Proj\left(\bigoplus_{n \in \mathbb{N}}\left(H^0\left(\calX_p,nd(K_{\calX_p}+\mathbf{v}(t_i)\calD_p\right)\right)\right).
		\end{equation*}
		is the canonical model of $(\calX_p, \bfv(t)\calD_p)$. We conclude moreover that the rational map $h: \calX \dashrightarrow \calY$ is in fact a morphism as it basechanges to the morphism (\ref{eqn:lc}) via the surjective map $B \to \sM$, and that $h$ induces the fiberwise canonical model for each $p \in \sM$.

		We now need to produce a family of divisors $\calD_\calY$ (see Section \ref{section:moduli:of:stable:pairs}). On $\calX$ we have $n$ well-defined families of divisors $\left(\calX,\calD^{(i)}\right) \to \sM$ and a universal pair $\left(\calX,\sum (s a_i + (1-s)b_i) \calD^{(i)}\right) \to \sM$. 
		We wish to define $\calD_\calY^{(i)}$ as the pushforward  to get $\calD^{(i)}_\calY:=h_*\left(\calD^{(i)}\right)$ to $\calY$. 
		We need to check that $\left(\calY,\calD^{(i)}_\calY\right)$ is a well-defined family of divisors for each $i$
		\cite[Definition 4.2]{kollarmodulibook} (see also Section \ref{section:moduli:of:stable:pairs}), that is,
		\begin{enumerate}[(a)]
			\item $\calY \to \sM$ is flat with deminormal fibers;
			\item for every $i$, $\Supp\left(\calD^{(i)}_{\calY}\right) \to \sM$ is equidimensional of dimension $n-1$;
			\item for every $p \in \sM$, the fiber $\calY_p$ is smooth at the generic point of $\Supp\left(\calD^{(i)}_{\calY}\right)$, and
			\item the assumptions of \cite[Corollary 4.5 (2)]{kollarmodulibook} (which we recall below) apply.
		\end{enumerate}
		
		Each of the statements (a), $\ldots$, (d) can be checked \'etale locally, so we can pull back to an \'etale cover $U \to \sM$ by a scheme. We denote by $f : X \to Y$ the pull-back of $h$ and let $D, D^{(i)}$ and $D_Y^{(i)}$ be the divisorial pull-backs of $\calD$, $\calD^{(i)}$ and $\calD_Y^{(i)}$ respectively as defined in Section \ref{section:moduli:of:stable:pairs}. Note moreover that $f_*D^{(i)} = D_Y^{(i)}$ as taking divisorial part and scheme theoretic image both commute with \'etale base change. The situation is summarized in the following diagram. 
		
		$$
		\xymatrix{ \Supp\left(D^{(i)}\right) \ar[r] &X \ar[r] \ar[d]_f & \calX \ar[d]^h \\ \Supp\left(D_Y^{(i)}\right)\ar[r]&Y \ar[r] \ar[d]_g & \calY \ar[d] \\& U \ar[r] & \sM}
		$$
		
		Recall now that for every $u \in U$, we have:
		\begin{itemize}
			\item The morphism $f_u:X_u \to Y_u$ is the stable model of $$\left(X_u, \sum_j (t_ia_j + (1-t_i)b_j)D_u^{(j)}\right)$$, and
			\item $\Supp \left(D^{(i)}\right)_u=\Supp\left(D_u^{(i)}\right)$ since $\left(X,D^{(i)}\right) \to U$ is a family of pairs.
		\end{itemize}
		In particular, for each $u \in U$, the fiber $Y_u$ comes with $n$ divisors (namely $(f_u)_*(D^{(i)}_u)$). \\
		
		\noindent \textbf{Claim:} For each $u \in U$, we have
		$$
		\Supp\left(f_*D^{(i)}\right)_u\eqdef \Supp\left((f_u)_*D^{(i)}_u\right).
		$$
		
		We thank the referee for suggesting a more elegant proof of this claim. \\
		
		\noindent\emph{Proof of claim:} To prove the claim, we may replace $D^{(i)}$ with an irreducible component, and since we work on one $i$ at a time, we simply denote it by $D$. The key observation is that $D\to f(D)$ is generically finite if and only if $f(D) \to U$ is equidimensional of relative dimension $\dim D_u$. Indeed, the dimensions of the fibers of both $f(D)\to U$ and $D\to f(D)$ are both upper semicontinuous, and the sum
		$$
		\dim f(D)_u + (\dim D_u - \dim f(D)_u) = \dim D_u
		$$
		is constant by assumption. Thus $f(D) \to U$ is equidimensional and $$\dim D_u - \dim f(D)_u$$ is constant. Note that $f$ is also fiberwise birational for general $u \in U$, so we conclude that $D \to f(D)$ is generically finite if and only if $f_*D = f(D)$, if and only if $(f_u)_*(D_u) = f_u(D_u)$ for general $u \in U$, if and only if $\Supp((f_u)_*(D_u)) = f_u(D_u)$ for all $u \in U$. In this case, it follows that 
		$$
		\Supp\left((f_u)_*D^{(i)}_u\right) = f_u(D^{(i)}_u) = f(D^{(i)})_u = \Supp\left(f_*D^{(i)}\right)_u. 
		$$
		Otherwise, $f_*D = 0$, and so $(f_u)_*(D_u) = 0$ for all $u \in U$ as well. In either case, the claim holds.

		Now we are ready to check conditions (a), $\ldots$, (d).  For (a), it suffices to check the conditions after pull back along all morphisms $\xi:\Spec(R) \to \sM$ from the spectrum of a DVR by the valuative criterion for flatness (see also e.g. \cite[Lemma 10.58]{kollarmodulibook}). Now the construction of $\calY \to \sM$ via the relative Proj commutes with basechange so the pull-back is the canonical model of a locally stable family over a DVR which is flat with deminormal fibers by the construction of stable limits.

		Next, (b) and (c) are properties of the fibers over points, so can be checked for each $u \in U$. Thus, they follow from the claim that
		$$
		\Supp(f_*D^{(i)})_u = \Supp((f_u)_*D_u^{(i)})
		$$
		and the fact that $\left(Y_u,\sum_j (t_ia_j+(1-t_i)b_j)(f_u)_*D_u^{(j)}\right)$ is a stable pair. 
		
		We now show (d), using \cite[Corollary 4.5]{kollarmodulibook}.
		With the notations of the previous paragraphs, we need to show the following.
		Consider $\nu:U^\nu \to U$ the normalization, and let $X_n$ (resp. $D_n^{(i)}$, $Y_n$ and $f_n$) be the pull-back of $X$
		(resp. $D^{(i)}$, $Y$ and $f$). Then $(Y_n,(f_n)_*(D_n^{(i)}))$ is a well-defined family of pairs from \cite[Theorems 4.3 \& 4.4]{kollarmodulibook}. We need to show that for every two points  $u,v \in U$ with
		$\nu(u)=\nu(v)$, we have $(Y_{u},((f_n)_*(D_n^{(i)}))_{u})=(Y_{v},((f_n)_*(D_n^{(i)}))_{v})$.
		But from the claim above, we know that $\Supp((f_n)_*(D_n^{(i)}))_u$ is the support of the push forward of $(D_n^{(i)})_u$, via the map that takes the stable model of $(X_u,\sum_j (t_ia_j + (1-t_i)b_j) (D^{(i)}_n)_u)$. In particular, it is uniquely determined by $(X_u,\sum_j (t_ia_j + (1-t_i)b_j) (D^{(i)}_n)_u)$.
		But since $\nu(s)=\nu(t)$ and the family $(X_n,\sum_j (t_ia_j + (1-t_i)b_j)D_n^{(j)})$ is pulled back via $\nu$, we have 
		$$(X_u,\sum_j (t_ia_j + (1-t_i)b_j) (D^{(i)}_n)_u)=(X_v,\sum_j (t_ia_j + (1-t_i)b_j) (D^{(i)}_n)_v).$$
		
		Putting this together, we conclude that $\sM$ carries a canonical $\bfv(t)$-weighted stable family \\ $(\calY, h_*\bfv(t)\calD)$ which induces the required morphism $\sM \to \sM_t$.
		
		Finally, as we showed above, the formation of the log canonical morphism 
		$$
		f : X \to \operatorname{Proj}_B \ R(X/B, K_{X/B} + \bfv(t)D)
		$$ of the log canonical ring commutes with basechange for all $\bfv(s)$-weighted stable families parametrized by $\sM$ and similarly the formation of the Weil divisor $\mathrm{Weil}(f_*D)$ also commutes. Therefore, the resulting morphism $\sM \to \sM_t$ can be described pointwise as taking a point $p$ corresponding to the stable pair $(X_p, \bfv(s)D_p)$ to the point of $\sM_t$ classifying the stable pair $$(\operatorname{Proj} R(K_{X_p} + \bfv(t)D_p), f_*\bfv(t)D).$$ In particular, over the locus where $X$ is normal, the morphism is induced by taking the fiberwise canonical model. \end{proof} 
	
	\begin{remark} If $(X,D) \to B$ is a locally stable family of pairs with $B$ smooth, then the canonical model over $B$ is a stable family by \cite[Corollary 4.57]{kollarmodulibook}. The main difficulty in the above Theorem then is descending the conditions on a stable family along the non-smooth morphism $B \to \sM$.
	\end{remark} 
	
	The following Corollary will be useful in the proof of Theorem \ref{thm:reduction:morphism:birational:for:ZMT}.
	\begin{cor}
		\label{cor:reduction:morphisms:are:onto}
		Following the notation of Theorem \ref{theorem:flip-like:morphisms}, the morphisms $\beta_{t_{i+1}}$ and $\alpha_{t_{i}}$ are surjective.
	\end{cor}
	\begin{proof}We prove the desired statement for $\alpha_{t_i}$, the case of $\beta_{t_{i+1}}$ is analogous. 
		
		From Theorem \ref{thm:morphismextends} and Definition \ref{Def:calM}, we have a surjective morphism $p:B \to \sM_{(t_i, t_{i+1})}$ with $B$ a smooth projective variety,
		induced by the family $(Z_{s},\mathbf{v}(s)\Delta_{s})\to B$ for any $s \in (t_i, t_{i + 1})$. Then to show that $\alpha_{t_i}$ is surjective, it suffices to show that $\alpha_{t_i} \circ p$ is surjective.
		
		The composition $B \to \sM_{(t_i, t_{i+1})} \xrightarrow{\alpha_{t_i}} \sM_{t_i}$ is induced
		by taking the canonical model of the pair $(Z_{s},\mathbf{v}(t_i)\Delta_{s})$ over $B$, which from Proposition \ref{prop:MMP:with:scaling:Jakub} agrees with $(Z_{t_i},\mathbf{v}(t_i)\Delta_{t_i})$. Now, the desired statement follows from the definition of $\sM_{t}$.
	\end{proof}
	
	Finally, we end the section with a discussion of the name ``flip-like morphisms''.
	\begin{notation}
		When working around a single $(\bfa \to \bfb)$-wall $t_i$, we will denote $\sM_{(t_{i-1}, t_i)}$ (resp. $\sM_{(t_i,t_{i + 1})}$) by $\sM_{t_i - \varepsilon}$ (resp. $\sM_{t_i + \varepsilon})$.
	\end{notation}
	Theorem \ref{theorem:flip-like:morphisms} guarantees the existence of maps $\calX_{t_i - \varepsilon} \to \calX_{t_i} \leftarrow \calX_{t_i+ \varepsilon}$
	of universal families. These universal families lie over different moduli spaces. However, we can pull back the above diagram to the fiber product $\mathfrak{F} := \sM_{t_i - \varepsilon} \times_{\sM_{t_i}} \sM_{t_i + \varepsilon}$ to obtain a diagram
	$$
	\xymatrix{\calZ_{t_i - \varepsilon} \ar[rd] \ar[rdd] &  & \calZ_{t_i + \varepsilon} \ar[ld] \ar[ldd] \\ &\calZ_{t_i} \ar[d] & \\ &\mathfrak{F}&}
	$$
	which one can think of as a sort of \emph{universal generalized log flip} (see \cite[Proposition 8.4]{AB17} and the preceding discussion). Indeed, pulling back this diagram along the natural morphism $B \to \mathfrak{F}$ yields a generalized log flip over $B$.
	Here we say \emph{generalized} to emphasize the fact that the log canonical contraction $Z_{t_i + \varepsilon} \to Z_{t_i}$ can be the contraction of a higher dimensional extremal face and thus can contract both divisorial and higher codimension exceptional loci. 
	
	Theorem \ref{theorem:flip-like:morphisms} can be summarized then by saying that this universal generalized log flip induces flip-like morphisms
	$\sM_{t_i - \varepsilon} \xrightarrow{\beta_{t_i}} \sM_{t_i} \xleftarrow{\alpha_{t_i}} \sM_{t_i + \varepsilon}$. In the following sections, we will see that $\beta_{t_i}$ is in fact an isomorphism after passing to the normalizations of the moduli spaces.
	In the following sections, we will see that $\beta_{t_i}$ is in fact an isomorphism after passing to the normalizations of the moduli spaces. 
	
	\section{Quasi-finiteness of the flip-like morphism below a wall}\label{Section:finite} 
	
	The goal of this section is to prove that for any $(\bfa \to \bfb)$-wall $t_i$, the flip-like morphism $\beta_{t_i} : \sM_{t_i - \varepsilon} \to \sM_{t_i}$ of Theorem \ref{theorem:flip-like:morphisms} is quasi-finite.  
	
	\begin{theorem}\label{thm:finite}
		The morphism $\beta_{t_i}$ is quasi-finite.
	\end{theorem}

	To prove Theorem \ref{thm:finite}, we consider the following situation. Let $q \in \sM_{t_i}$ be a point corresponding to a stable pair $(X, \bfv(t_i)D_X)$. Each point $p \in \beta_{t_i}^{-1}(q)$ corresponds to a $\bfv(t_{i} - \varepsilon)$-weighted stable pair $(Y, \bfv(t_i - \varepsilon)D_Y)$. We need to show that there are finitely many such $Y$ given a fixed $(X, \bfv(t_i)D_X)$. To do this, we need to understand how the different models $(Y, \bfv(t_i - \varepsilon)D_Y)$ are related to $(X,  \bfv(t_i)D_X)$. This is accomplished by the following lemma. 
	
	\begin{lemma}\label{lemma:the:reduction:morphisms:contract:curves:in:D}
		Let $(Y,\mathbf{v}(t_i-\varepsilon)D)$ be an slc pair corresponding to  $p \in \sM_{t_i-\varepsilon}(k)$, and
		let the pair $(X,\mathbf{v}(t_i)D_X)$ be the image  $\beta_{t_i}(p)$. Then there is a morphism $h: Y \to X$ with the following properties:
		\begin{enumerate}
			\item A curve $C$ gets contracted by $h$ if and only if $(K_Y+\mathbf{v}(t_i)D) \cdot C=0$,
			\item $h$ has connected fibers,
			\item $\operatorname{Exc}(h) \subseteq \Supp(\mathbf{v}(t_i)D)$, in particular $h$ does not contract any component of $Y$, and
			\item $h^*(K_X+\mathbf{v}(t_i)D_X)=K_Y+\mathbf{v}(t_i)D.$
		\end{enumerate}
	\end{lemma}
	\begin{proof}

		By Theorem \ref{theorem:flip-like:morphisms} and the construction of $\beta_{t_i}$, $X$ is the $\operatorname{Proj}$ of the log canonical ring of $(X, \bfv(t_i)D)$, the \emph{a priori} rational map $h : X \to Y$ is a morphism, $D_X = h_*D$, and the formation of the $\operatorname{Proj}$ and $h_*D$ as a Weil divisor both commute with base change. If $Y$ is klt, then $X$ is klt and $(1)$, $(2)$ and $(4)$ follow from basic properties of the canonical model of log terminal model. 
		
		In general, every point of $\sM_{t_i - \varepsilon}$ is smoothable to a klt pair. Therefore, consider a one parameter family $(\sY,\mathbf{v}(t_i-\varepsilon)\sD) \to \Spec(R)$ in $\sM_{t_i-\varepsilon}$ with closed fiber isomorphic to $(Y,\mathbf{v}(t_i-\varepsilon)D)$ and generic fiber klt, and consider the relative canonical model of $(\sY,\mathbf{v}(t_i)D)$ over $\Spec(R)$, namely $(\sX,\mathbf{v}(t_i)\sD_{\sX}) \to \Spec(R)$. Then the pair $(X,\mathbf{v}(t_i)D_X)$ is the closed fiber of $(\sX,\mathbf{v}(t_i)\sD_{\sX}) \to \Spec(R)$ and the total spaces $\sX$ and $\sY$ are normal so from the construction of the canonical model, we have that:
		\begin{itemize}
			\item The morphism $\gamma:\sY \to \sX$ has connected fibers, and
			\item a curve $C \subseteq Y$ gets contracted by $\gamma$ if and only if $$C \cdot(K_\sY+\mathbf{v}(t_i)\sD)=C \cdot(K_Y+\mathbf{v}(t_i)D)=0.$$
		\end{itemize}In particular, we have shown (1). Moreover, since a fiber of $h: Y \to X$ is also a fiber of $\gamma: \sY \to \sX$, we have also shown (2).
		
		To prove (3), we only need to check that a curve $C$ which is not contained in $\Supp(\mathbf{v}(t_i)D)$, satisfies $(K_Y+\mathbf{v}(t_i)D) \cdot C>0$. Note that $(K_Y+\mathbf{v}(t_i-\varepsilon)D) \cdot C>0$ since the pair $(Y,\mathbf{v}(t_i-\varepsilon)D)$ is stable.
		Moreover, $(\mathbf{v}(t_i)D) \cdot C \ge (\mathbf{v}(t_i-\varepsilon)D)\cdot C$ since $C$ is not contained in $\Supp(\mathbf{v}(t_i)D)$. Therefore
		$$(K_Y+\mathbf{v}(t_i)) \cdot C \ge (K_Y+\mathbf{v}(t_i-\varepsilon)) \cdot C>0.$$
		
		Finally, to show (4), let us denote the closed point by $c \in \Spec(R)$. By \cite[Lemma 1.28]{Kollarsingmmp}, the pair $(\sX,\mathbf{v}(t_i)\calD_{\sX}+\sX_c)$ is the stable model of $(\sY,\mathbf{v}(t_i)\sD+\sY_c)$. In particular, there is a morphism $\gamma:\sY \to \sX$ which restricts to $h$,  such that $\gamma^*(K_{\sX} +\mathbf{v}(t_i)\sD_{\sX}+\sX_c)=K_{\sY}+\mathbf{v}(t_i)\sD+\sY_c$. But
		$$(K_{\sX}+ \mathbf{v}(t_i)\sD_{\sX}+\sX_c)_{|\sX_c}=K_X+\mathbf{v}(t_i)D_X$$ and $(K_\sY+\mathbf{v}(t_i)\sD+\sY_c)_{|\sY_c}=K_Y+\mathbf{v}(t_i)D$, so (4) follows from the commutative diagram below, and functoriality of pull back:
		$$\xymatrix{Y \ar[r] \ar[d]_h & \sY \ar[d]^\gamma \\ X\ar[r]& \sX.}$$\end{proof}
	
	Our task now is to show that given $(X, \bfv(t_i)D_X)$, there are finitely many $(Y, \bfv(t_i - \varepsilon)D)$ as in Lemma \ref{lemma:the:reduction:morphisms:contract:curves:in:D}. In fact, it suffices to show there are countably many. The following lemma allows us to normalize $(X, \bfv(t_i)D_X)$ and reduce to the log canonical case. 
	
	\begin{lemma}\label{Lemma:red:lc:case}
		Let $(X,\bfv(t)D)$ be a $\bfv(t)$-weighted stable pair corresponding to a point $q \in \calM_t(k)$ and let $\nu:X^\nu\to X$ be the normalization of $X$, with conductor divisor $\overline{\Gamma} \subset X^\nu$. Assume that there are countably many log canonical  pairs $(Y,\bfv(t)D_Y+\Gamma_Y)$ such that 
		\begin{itemize}
			\item $(Y,\bfv(t-\varepsilon)D_Y+\Gamma_Y)$ is stable, 
			\item $K_Y + \bfv(t) D_Y+\Gamma_Y$ is semiample, and
			\item the canonical model of $(Y,\bfv(t) D_Y+\Gamma_Y)$ is $(X^\nu, \bfv(t)D + \overline{\Gamma})$.
		\end{itemize}
		Then the fiber of $\beta_{t}^{-1}(q)$ is countable.
	\end{lemma}
	\begin{proof}
		Let $(Z,\bfv(t_i - \varepsilon)D_Z)$ be a pair in the fiber of $\beta_{t_i}$ at $p$, and let $$(Z^\nu,\bfv(t_i - \varepsilon)D_Z^\nu + \overline{\Gamma}_Z)$$ be its normalization. We need to show two claims: 
		\begin{enumerate}[(1)] 
			\item $K_{Z^\nu} +  \bfv(t_i)D_Z^\nu + \overline{\Gamma}_Z$ is semiample and the canonical model of
			$$(Z^\nu,\bfv(t_i)D_Z^\nu + \overline{\Gamma}_Z)$$ is $(X^\nu, \bfv(t)D + \overline{\Gamma})$, and
			\item there are only finitely many stable pairs with a given normalization.
		\end{enumerate}
		Claim $(1)$ and the assumption imply there are countably many pairs $$(Z^\nu,\bfv(t_i - \varepsilon)D_Z^\nu + \overline{\Gamma}_Z)$$ which could be the normalization of the pair $(Z,\bfv(t_i - \varepsilon)D_Z)$. Claim $(2)$ follows from \cite[Theorem 5.13]{Kollarsingmmp}. For claim $(1)$, we first produce a morphism as below, using the universal property of the normalization (see \cite[Tag 0BB4]{stacks-project}):
		
		$$
		\xymatrix{(Z^\nu,\mathbf{v}(t_i)D^\nu_Z+\overline{\Gamma}_Z) \ar@{.>}[r]^-{\gamma^\nu} \ar[d]_\nu & (X^\nu, \mathbf{v}(t_i) D^\nu+ \overline{\Gamma}) \ar[d]^\mu 
			\\ 
			(Z,\mathbf{v}(t_i) D_Z) \ar[r]^\gamma & (X, \mathbf{v}(t_i)D)}
		$$
		
		We need to check that the composition $Z^\nu \xrightarrow{\nu} Z \xrightarrow{\gamma} X $ does not contract any irreducible components.
		This follows since $\nu$ is finite and $\gamma$ does not contract any irreducible component by Lemma \ref{lemma:the:reduction:morphisms:contract:curves:in:D}. By Lemma, \ref{lemma:the:reduction:morphisms:contract:curves:in:D} we also have that $\gamma^*(K_X+\bfv(t_i)D) = K_Z+\bfv(t_i )D_Z$. Moreover,
		since $\mu$ and $\nu$ are normalizations, $$\nu^*(K_Z+\bfv(t_i)D_Z) = K_{Z^\nu}+\bfv(t_i)D_Z^\nu + \overline{\Gamma}_Z, \text{ and }
		$$
		$$
		\mu^*(K_X + \bfv(t_i)D) = K_{X^\nu} + \bfv(t_i)D^\nu + \overline{\Gamma}.
		$$
		Then from the commutativity of the diagram above we have $$(\gamma^\nu)^*(K_{X^\nu} + \bfv(t_i)D^\nu + \overline{\Gamma}) = K_{Z^\nu} + \mathbf{v}(t_i)D^\nu_Z+\overline{\Gamma}_Z.$$
		The latter is semiample as it is the pull back of an ample divisor. To check that $\gamma^\nu$ is the canonical model, it suffices to check that a curve $C$ is contracted by $\gamma^\nu$ if and only if $C\cdot(K_{Z^\nu} + \mathbf{v}(t_i)D^\nu_Z+\overline{\Gamma}_Z)=0$.
		Since the normalizations $\nu$ and $\mu$ do not contract curves, we have:
		\begin{align*}
			\gamma^\nu \text{ contracts }C &\Longleftrightarrow \mu \circ \gamma^\nu \text{ contracts }C \\ &\Longleftrightarrow \gamma \circ \nu \text{ contracts }C \Longleftrightarrow \gamma \text{ contracts } \nu(C).
		\end{align*}
		From Lemma \ref{lemma:the:reduction:morphisms:contract:curves:in:D}, the morphism $\gamma$ contracts $\nu(E)$ if and only if 
		$$
		0=(K_Z+\mathbf{v}(t_i)D_Z)\cdot\nu_*(C)=(K_{Z^\nu} + \mathbf{v}(t_i)D^\nu_Z+\overline{\Gamma}_Z)\cdot C,
		$$ 
		as desired.
		
		To conclude the statement, observe that by \cite[Theorem 5.13]{Kollarsingmmp}, a stable pair $(Z,\bfv(t_i-\varepsilon)D_Z)$ is uniquely determined by its normalization $$(Z^\nu,\bfv(t_i-\varepsilon)D_Z + \overline{\Gamma}_Z)$$, and an involution on the \textit{different} $(\overline{\Gamma}^\nu_Z, \Diff_{\Gamma^\nu_Z}(D))$. The latter is a stable pair by adjunction, and so has finitely many automorphisms by \cite[Proposition 5.5]{KP17}.
	\end{proof}
	We are left with proving the following result, which was communicated to us by J\'anos Koll\'ar. We thank him for allowing us to include it here.
	
	\begin{prop}\label{prop:lc:case:finite}
		Assume that $(X,D + \bar{\Gamma})$ is a normal stable pair, and let $I$ be a finite set of positive rational numbers. Consider the set of pairs $(Y,D_Y + \overline{\Gamma}_Y)$ such that:
		\begin{enumerate}
			\item $(Y,(1-\varepsilon)D_Y + \overline{\Gamma}_Y)$ is stable for every $0<\varepsilon < \varepsilon_0$,
			\item The canonical model of $(Y,D_Y + \overline{\Gamma}_Y)$ is $(X,D + \bar{\Gamma})$, and
			\item The coefficients of $D_Y + \overline{\Gamma}_Y$ are in $I$.
		\end{enumerate}
		Then this set is countable.
	\end{prop}
	\begin{proof}Let $p:Y\to X$ be the morphism which gives the stable model of $(Y,D_Y + \overline{\Gamma}_Y)$. 
		First observe that, by a theorem of Matsusaka--Mumford applied to the variety $Y$ and the divisor $K_Y + (1-\varepsilon)D_Y + \overline{\Gamma}_Y$ (see \cite[Theorems 11.39 \& 11.40]{kollarmodulibook}), the pair $(Y,D_Y + \overline{\Gamma}_Y)$ is uniquely determined by $(X,D)$, the coefficients $I$, and the divisors extracted by $Y\to X$ with their discrepancy. These divisors have strictly negative discrepancy for $(X,D)$. Indeed, by Lemma \ref{lemma:the:reduction:morphisms:contract:curves:in:D} the exceptional locus is contained in $D_Y+\overline{\Gamma}_Y>0$ and $$p^*(K_X+D + \overline{\Gamma})=K_Y+D_Y + \overline{\Gamma}_Y = K_Y + p_*^{-1}(D + \overline{\Gamma}) - \sum_{E\text{ is }p-\text{exceptional}} a(E,X,D)E.$$
		Then it suffices to show that there are countably many divisors with strictly negative discrepancy for $(X,D + \overline{\Gamma})$. Indeed, let $Z\to X$ be a log-resolution of $(X,D + \overline{\Gamma})$. As the map $q:Z\to X$ extracts finitely many divisors, it suffices to check that there are finitely many divisors with negative discrepancy for the pair $(Z,D_Z + \overline{\Gamma}_Z)$, where $D_Z + \overline{\Gamma}_Z$ satisfies $K_Z+D_Z + \overline{\Gamma}_Z = q^*(K_X + D + \overline{\Gamma})$ (so $D_Z + \overline{\Gamma}_Z$ no longer needs to be effective). This follows from Lemma \ref{lemma:negative:discrepancies:logsmooth:pair} below, as such divisors can only be extracted by repeatedly blowing up some strata of the snc pair $(Z,D_Z + \overline{\Gamma}_Z)$ (Definition \ref{def:strata} below).
	\end{proof}
	
	
	\begin{definition}\label{def:strata}(see \cite[Definition 11.10]{kollarmodulibook})
		Let $(X,D = \sum_{i \in I}a_i D_i)$ be a simple normal crossing (snc) pair. A \emph{stratum} of $(X,D)$ is any irreducible component of an intersection $\cap_{i \in J} D_i$ for some $J \subset I$. \end{definition}

	\begin{lemma}\label{lemma:negative:discrepancies:logsmooth:pair}
		Let $(X,D)$ be an lc pair, with $D$ a priori not effective, such that $(X,\Supp(D))$ is snc. Let $E \subseteq W$ be a divisor on a birational model $W \to X$, with $a(E;X,D)<0$. Let $R:= \calO_{W,E}$ be the local ring at the generic point of $E$ in $W$, and let $\xi$ be the closed point of $\Spec(R)$. Then there is a sequence of blow-ups $X_m \xrightarrow{p_m} X_{m+1} \xrightarrow{p_{m-1}} ... \xrightarrow{p_2} X_1:=X$ so that, if we denote with $D_i:=p_i^*(D_{i-1})$ and with $q_i$ the image of $\xi$ through the morphism $\Spec(R) \to X_i$, then:
		\begin{enumerate}
			\item $(X_i,\Supp(D_i))$ is snc;
			\item $X_{i} \to X_{i-1}$ is the blow up of a stratum of $X_{i-1}$, and
			\item $q_m$ has codimension one in $X_m$.
		\end{enumerate}
	\end{lemma}
	\begin{proof}Let $v$ be the valuation associated to $R$.
		Now, \cite[Lemma 2.45]{KM98} gives us a recipe for producing a sequence of blow-ups $X_m \to X_{m-1} \to ... \to X_0$ so that $E$ is a divisor in $X_m$. In particular, each morphism $X_i \to X_{i-1}$ is the blow-up of the closure of $q_{i-1}$.
		
		Therefore, since if we blow-up a stratum in a log-smooth pair, we still get a log-smooth pair, it suffices to proceed by induction showing that:
		\begin{itemize}
			\item The closure of $q_i$ is a stratum in $X_i$, and
			\item For every divisor $F$ over $X$, we have $a(F;X_i,D_i)=a(F;X_{i+1},D_{i+1})$.
		\end{itemize}
		The first claim follows from \cite[2.10.1]{Kollarsingmmp} and the next line, the second from \cite[Lemma 2.30]{KM98}.
	\end{proof}
	
	\begin{proof}[Proof of Theorem \ref{thm:finite}]By Lemma \ref{Lemma:red:lc:case} and Proposition \ref{prop:lc:case:finite}, the fibers of $\beta_{t_i}$ are countable. Since$\beta_{t_i}$ is of finite type, it follows that the fibers are finite, so $\beta_{t_i}$ is quasi-finite.
	\end{proof}
	
	\section{Reduction morphisms up to normalization}\label{Section:reduction} 
	
	The goal of this section is to construct reduction morphisms $\rho_{\bfa, \bfb}$ for weight vectors $\bfb \le \bfa$ generalizing Hassett's reduction morphisms \cite[Theorem 4.1]{Hassett} to higher dimensions. To accomplish this, we need to normalize the moduli space (see Section \ref{sec:normalization} for an example showing this is necessary). 
	
	\begin{definition}\label{Def:calN} In the setting of Definition \ref{Def:calM}, we let $\sN_t$  for $t \in [0,1]$ denote the normalization of $\sM_t$. We denote by $\sN_\bfa$ (resp. $\sN_\bfb$) the normalization of $\sM_\bfa$ (resp. $\sM_{\bfb}$).  
	\end{definition}

	\begin{theorem}\label{theorem:reduction:are:iso} Let $\bfb \le \bfa$ be weight vectors and $0 < t_1 < \ldots < t_m = 1$ the $(\bfb \to \bfa)$-walls. Then for any $t_i$, the flip-like morphism $\beta_{t_i} : \sM_{t_i - \varepsilon} \to \sM_{t_i}$ induces an isomorphism $\beta_{t_i}^\nu:\sN_{t_i-\varepsilon} \to \sN_{t_i}$. \end{theorem}
	
	The proof of Theorem \ref{theorem:reduction:are:iso} proceeds as follows:
	\begin{enumerate}
		\item $\beta_{t_i}$ is quasi-finite by Theorem \ref{thm:finite} in the previous section, and
		\item $\beta_{t_i}$ is proper, representable, and an isomorphism on a dense open subset (see Theorem \ref{thm:reduction:morphism:birational:for:ZMT} below).
	\end{enumerate}
	Then Theorem \ref{theorem:reduction:are:iso} follows then from Zariski's main theorem.
	
	\begin{definition}\label{def:reduction}
		Composing $(\beta_{t_i}^\nu)^{-1} : \sN_{t_i} \to \sN_{t_i-\varepsilon}$ with $\alpha^\nu_{t_{i-1}} : \sN_{t_{i-1} + \varepsilon} \to \sN_{t_{i - 1}}$ for all $i$ gives the desired reduction morphisms:
		$$
		\begin{array}{c}
			\rho_{\bfb,\bfa}:\sN_{\mathbf{a}} \to \sN_{\mathbf{b}} \\ \rho_{\bfb, \bfa} := \alpha_{t_0}^\nu \circ (\beta_{t_1}^\nu)^{-1} \circ \alpha_{t_1}^\nu \circ \ldots \circ \alpha_{t_{m-1}}^\nu \circ \beta_{t_m}^{-1}.
		\end{array}
		$$
	\end{definition}
	
	\begin{remark} Note that both $\alpha_{t_i}$ and $\beta_{t_i}$ are dominant by Corollary \ref{cor:reduction:morphisms:are:onto}, so they induce morphisms $\alpha_{t_i}^\nu$ and $\beta_{t_i}^\nu$ between normalizations. \end{remark}
	
	\begin{remark}\label{rem:composition} For any weight vector $\bfv(t) = t\bfa + (1-t)\bfb$, the reduction morphisms are compatible by definition: $\rho_{\bfb, \bfv(t)} \circ \rho_{\bfv(t), \bfa} = \rho_{\bfb, \bfa}$. In general, for weight vectors $\bfc \le \bfb \le \bfa$ that are not co-planar we may have
		$$
		\rho_{\bfc, \bfb} \circ \rho_{\bfb, \bfa} \neq \rho_{\bfc, \bfa}. 
		$$
		This is because the construction of the moduli spaces and morphisms a priori depends on the MMP with scaling we used to get from weights $\bfa$ to $\bfb$. In section \ref{section:countereg}, we will give some examples showing this can occur and state conditions under which the reduction morphisms are compatible for all weights (Theorem \ref{thm:last:thm}). 
	\end{remark}
	
	\begin{theorem}\label{thm:reduction:morphism:birational:for:ZMT}
		The morphism $\beta_{t_i} : \sM_{t_i - \varepsilon} \to \sM_{t_i}$ is representable, proper and birational.
	\end{theorem}
	\begin{proof}
		Recall that by Theorem \ref{thm:morphismextends},  the morphism $\beta : \sM_{t_i-\varepsilon} \to \sM_{t_i}$ can be described as follows on the dense open subset parametrizing klt pairs. Given a point $p \in \sM_{t_i-\varepsilon}(k)$ corresponding to a stable pair $(Y, \bfv(t_i -\varepsilon)D)$, $\beta_{t_i}(p)$ classifies the canonical model of the pair $(Y, \bfv(t_i)D)$ which we will denote by $(X, \bfv(t_i)D_X)$. 
		
		The morphism $\beta_{t_i}$ is proper since the source and target are proper, and it is surjective by Corollary \ref{cor:reduction:morphisms:are:onto}. It is generically injective, since we can recover
		$(Y,\mathbf{v}(t_i-\varepsilon)D)$ from the pair $(X,\mathbf{v}(t_i)D_X)$ in the klt case. Indeed, by Proposition \ref{prop:MMP:with:scaling:Jakub} the pair $(Y,\mathbf{v}(t_i-\varepsilon)D)$ is the log canonical  model of $(X,\mathbf{v}(t_i-\varepsilon)D_X)$ when $(X, \mathbf{v}(t_i)D_X)$.
		
		To show representability, consider a stable pair $(Y,\mathbf{v}(t_i-\varepsilon)D)$ which corresponds to a point $p \in \sM_{t_i-\varepsilon}(k)$, and suppose
		$\tau$ is an automorphism of the pair. Let $(X,\mathbf{v}(t_i)D_X)$ be the pair corresponding to $\beta_{t_i}(p)$. Then $\tau$ induces an automorphism $\tau_X$ of $(X,\mathbf{v}(t_i)D_X)$ by functoriality of the construction of $\beta_{t_i}$. We need to prove that $\tau_X=\operatorname{Id} \Longrightarrow \tau=\operatorname{Id}$. This is proved as in \cite[Observation 8.4]{Inc18}, using Lemma \ref{lemma:the:reduction:morphisms:contract:curves:in:D}.
		
		We are left with showing that $\beta_{t_i}$ is birational. To do this, we produce a dense open substack $\calU \subseteq \sM_{t_i}$ such that 
		$
		\beta_{t_i} : \beta_{t_i}^{-1}(\calU) \to \calU
		$
		is an isomorphism.

		First, observe that if we denote with roman letters the coarse moduli spaces, the induced morphism $\beta_{t_i}^c:M_{t_i-\varepsilon} \to M_{t_i}$ is proper as the source is proper and the target is separated. Moreover, $\beta_{t_i}^c$ is quasi-finite, surjective and generically injective, since these are properties
		that we can check on algebraically closed points, they hold for $\beta_{t_i}:\sM_{t_i-\varepsilon} \to \sM_{t_i}$, and we have a bijection $|\sM_t(\Spec(k))| \cong |M_t(\Spec(k))|$ for $k$ algebraically closed. Since we are in characteristic $0$, $\beta_{t_i}^c$ is also generically unramified. Therefore, by Zariski's main theorem, the morphism $\beta_{t_i}^c$ induces an isomorphism of normalizations.
		
		As $\sM_{t_i}, \sM_{t_i-\varepsilon}$, $M_{t_i}$ and $M_{t_i-\varepsilon}$  are reduced and irreducible, there exist normal dense open subsets of all these spaces. Moreover, these can be chosen to fit in the following diagram  
		
		$$\xymatrix{\calV \ar@/^1.5pc/[rrr] \ar[d] \ar[r] & \sM_{t_i-\varepsilon} \ar[d] \ar[r]_{\beta_{t_i}} & \sM_{t_i} \ar[d] & \ar[l] \ar[d] \calU \\ V \ar@/_1.5pc/[rrr] \ar[r] & M_{t_i-\varepsilon} \ar[r]^{\beta^c_{t_i}} & M_{t_i} & \ar[l] U}$$\\
		
		\noindent where the left and right squares are cartesian. Since normalization is an isomorphism on the locus that is already normal, then up to shrinking the open subsets, we can summarize the situation as follows: 
		\begin{itemize}
			\item $U \subset M_{t_i}$ and $V = (\beta_{t_i}^c)^{-1}(U) \subset M_{t_i - \varepsilon}$ are open and dense subsets such that $(\beta_{t_i}^c)_{|V}:V\to U$ is an isomorphism,
			\item $\calV$ (resp. $\calU$) is a normal dense open substack of $\sM_{t_i}$ (resp. $\sM_{t_i - \varepsilon}$), with coarse space $U$ (resp. $V$), and
			\item $\calU$ and $\calV$ are contained in the locus parametrizing klt pairs.
		\end{itemize}
		
		Then the restriction $\beta_{t_i}\big|_{\calV} : \calV \to \calU$ is a representable morphism between normal Deligne--Mumford stacks which is an isomorphism on coarse moduli spaces. We wish to show that this is an isomorphism on the level of stacks. By construction, it induces an isomorphism of coarse spaces and hence a bijection on geometric points. Moreover, since we are in characteristic $0$, then up to shrinking $\calU$ and $\calV$, we may assume that $\beta_{t_i}\big|_{\calV}$ is \'etale. Thus, by \cite[Lemma 3.1]{alpkr}, it suffices to show that the morphism is stabilizer preserving.
		
		As we already know the morphism is representable, we are left with showing surjectivity of automorphism groups. In the notation above, for $p \in \calV(k)$ classifying a klt stable pair $(Y,\mathbf{v}(t_i-\varepsilon)D)$, we need to show that any automorphism of the pair $(X, \bfv(t_i)D_X)$ comes from an automorphism of $(Y,\mathbf{v}(t_i-\varepsilon)D)$. Recall that by construction of $\beta_{t_i}$, $(X, \bfv(t_i)D_X)$ is the canonical model of $(Y, \bfv(t_i)D)$. On the other hand, by Proposition \ref{prop:MMP:with:scaling:Jakub}, $(Y, \bfv(t_i-\varepsilon)D)$ is the canonical model of $(X, \bfv(t_i-\varepsilon)D_X)$ and there is a small birational morphism $Y \to X$.  
		
		In particular, there is an open subscheme $W \subseteq Y$ such that:
		\begin{enumerate}
			\item $W \to X$ is an isomorphism with its image (which we denote with $W_X$);
			\item $W$ is $\mathbb{Q}$-factorial, and
			\item The complement of $W$ and $W_X$ have codimension at least 2 in $Y$ and $X$ respectively.
		\end{enumerate}
		Therefore, $Y=\Proj(\bigoplus_n H^0(W_X,nd(K_{W_X}+\mathbf{v}(t_i-\varepsilon)(D_X)_{|W_X})))$ for $d$ divisible enough. Recall that $\mathbf{v}(t_i)D$ is shorthand for 
		$$
		\sum (t_ia_i+(1-t_i)b_i)D^{(i)}),
		$$
		so any automorphism of $(X,\mathbf{v}(t_i)D)$ fixes the components $D^{(i)}$. In particular, it induces an automorphism of $X$ which sends $D^{(i)}$ to itself. Thus, it induces an automorphism of 
		$$
		Y=\Proj(\bigoplus_n H^0(W_X,nd(K_{W_X}+\mathbf{v}(t_i-\varepsilon)(D_X)_{|W_X}))),
		$$ and since it preserves the $D^{(i)}$ it induces an automorphism of the pair $$(Y,\mathbf{v}(t_i-\varepsilon)D),$$ completing the proof.
	\end{proof}
	\begin{proof}[Proof of Theorem \ref{theorem:reduction:are:iso}] This follows from Theorem \ref{thm:finite}, Theorem \ref{thm:reduction:morphism:birational:for:ZMT} and Zariski's main theorem for representable morphisms of algebraic stacks \cite[Theorem 16.5]{LMB00}. \end{proof}
	
	Next we show that under some natural assumptions, the flip-like morphism $\alpha_{t_i} : \sM_{t_i + \varepsilon} \to \sM_{t_i}$ is also birational.  
	
	\begin{prop}\label{prop_answer_Zsolt_question} Let $t_i$ be an $(\bfa \to \bfb)$-wall and suppose that there exists a dense open substack $\calU \subset \sM_{t_i+\varepsilon}$. Denote by $(\calX_\calU,\bfv(t_i + \varepsilon)\calD_\calU)\to \calU$ the universal $\bfv(t_i+\varepsilon)-$weighted stable family over $\calU$. Suppose that the family $(\calX_\calU,\bfv(t_i)\calD_{\calU})\to \calU$ is also a stable. Then up to shrinking $\calU$, we have that
		\begin{itemize}
			\item $\alpha_{t_i}(\calU)$ is a dense open substack of $\sM_{t_i}$,
			\item $(\alpha_{t_i})_{|\calU}:\calU \to \alpha_{t_i}(\calU)$ is an isomorphism, and
			\item the pull-back of $(\calX_\calU, \bfv(t_i)\calD_\calU) \to \calU$ along $(\alpha_{t_i})^{-1}$ is the universal $\bfv(t_i)-$weighted stable family over $\alpha_{t_i}(U)$. 
		\end{itemize} In particular, $\alpha_{t_i} : \sM_{t_i+\varepsilon} \to \sM_{t_i}$ is birational.
	\end{prop}
	
	\begin{remark}\label{remark:alpha:not:finite} Note that in contrast to $\beta_{t_i}$, the flip-like morphism $\alpha_{t_i}$ is not finite in general. This happens already in the case of weighted stable curves \cite{Hassett} where $\alpha_{t_i}$ can contract high dimensional loci parametrizing rational tails with $n \geq 4$ special points which are contracted to a point by the canonical model for coefficients $\bfv(t_i)$. 
	\end{remark}
	
	\begin{remark}\label{remark:alpha:often:birational} The hypothesis of Proposition \ref{prop_answer_Zsolt_question} is often satisfied in practice. For example (see also Remark \ref{rem:situations}), one often begins with a family of pairs over an open base $U$ which are stable for all $t \in [0,1]$ and asks how the stable pairs compactification changes as $t$ varies. In this case, the image of $U$ inside $\sM_{t}$ is constructible by Chevalley's Theorem and dense by construction. Therefore, the image of $U$ contains a dense open substack of $\sM_t$. By assumption, the pairs over this substack are stable for all $t \in [0,1]$ so the hypothesis of the proposition is satisfied. 
	\end{remark}
	
	\begin{proof} We will denote $\alpha_{t_i}$ by $\alpha$ for convenience. 
		
		Let $f : U \to M_{t_i}$ be the coarse space map of the restriction $\alpha_{|\calU}:\calU \to \sM_{t_i}$. Then $f$ is dominant since $\calU$ is dense in $\sM_{t_i + \varepsilon}$ and $\alpha$ is surjective by Corollary \ref{cor:reduction:morphisms:are:onto}. Therefore, the image $f(U)$ is dense and constructible so by e.g. \cite[Chapter 2, Exercises 3.18-3.19]{Har77} there exists a dense open subset 
		$$
		V \subset f(U) \subset M_{t_i}. 
		$$
		Then, defining $\calV := \sM_{t_i} \times_{M_{t_i}} V$ and up to replacing $\calU$ with $\alpha^{-1}(\calV)$, we can assume that:\begin{enumerate}
			\item $\calU$ is an open and dense substack of $\sM_{t_i+\varepsilon}$,
			\item $\calV$ is an open and dense substack of $\sM_{t_i}$, and
			\item $\alpha(\calU) =\calV$ and $\alpha^{-1}(\calV) = \calU$.
		\end{enumerate}
		
		Now we proceed as in the proof of Theorem \ref{thm:reduction:morphism:birational:for:ZMT}. We know that the induced map on $k$-points $\calU(k) \to \calV(k)$ is injective, since by assumption, for any stable pair $(X, \bfv(t_i + \varepsilon)D)$ parametrized by $\calU$ the pair $(X, \bfv(t_i)D)$ is also stable. Moreover, it is surjective since $\alpha(\calU) = \calV$. Thus $\alpha_{|\calU}:\calU \to \calV$ is a morphism between normal separated stacks of finite type in characteristic $0$ which is a bijection on geometric points, so up to shrinking $\calU$ and $\calV$ further, we may assume that it is \'etale. Then by \cite[Lemma 3.1]{alpkr}, to show that $\alpha_{|\calU}$ is an isomorphism, it suffices to show it is stabilizer preserving. But this again follows by assumption since over $\calU$, the pair corresponding to a point $\alpha_{t_i}(p)$ is simply $(X, \bfv(t_i)D)$ where $p$ corresponds to $(X, \bfv(t_i+\varepsilon)D)$. \end{proof}
	
	\begin{cor}\label{cor:reductionbirational}
		Let $\bfb \le \bfa$, and let $f:(\calX,\calD_{\bf a})\to \sM_{\bfa}$ be the universal pair. Suppose that the restriction of $f$ to the generic point is $\bfb-$weighted stable. Then the reduction morphism $\rho_{\bf b,a} : \sN_\bfa \to \sN_\bfb$ is birational.
	\end{cor}
	\begin{proof}
		Since the normalization of a reduced stack is birational, it suffices to check the desired statement for the rational map $\sM_{\bf a}\dashrightarrow \sM_{\bf b}$. But according to Theorem \ref{theorem:flip-like:morphisms}, this factors as a finite sequence of flip-like morphisms $\sM_s \to \sM_t$. Such morphisms are birational by Theorem \ref{thm:reduction:morphism:birational:for:ZMT} and Proposition \ref{prop_answer_Zsolt_question}, so their composition is birational.
	\end{proof}
	
	We conclude the section with an application of representability of $\beta_t$ to isotrivial families.
	
	\begin{cor} Let $C$ be a smooth and irreducible curve and let $$f : (\calY, \bfv(t_i - \varepsilon)\calD) \to C$$ be a $\bfv(t_i - \varepsilon)$-weighted stable family. Suppose that the $\bfv(t_i)$-weighted stable model over $C$ is isomorphic to a product $(X, \bfv(t_i)D) \times C \to C$. Then the family $f : (\calY, \bfv(t_i - \varepsilon)\calD) \to C$ is also isomorphic to a product.
	\end{cor}
	\begin{proof} Let $\varphi : C \to \sM_{t_i- \varepsilon}$ be the moduli map induced by the stable family $f$. By assumption, the composition $\beta_{t_i} \circ \varphi : C \to \sM_{t_i}$ with $\beta_{t_i}$ is constant, that is, it factors through a closed point $x : \Spec k \to \sM_{t_i}$. Since $C$ is a smooth connected curve, $\varphi$ factors through a connected component of the reduced preimage $(\beta_{t_i}^{-1}(x))_{red}$. By Theorems \ref{thm:finite} and \ref{thm:reduction:morphism:birational:for:ZMT}, $\beta_{t_i}^{-1}(x)$ is a finite and representable over $\Spec k$. Therefore, $(\beta_{t_i}^{-1}(x))_{red}$ is a finite union of points and so $\varphi$ is the constant map. \end{proof}
	
	\section{Examples, counterexamples, and natural questions}\label{section:countereg}
	
	In this section, we discuss several natural generalizations of our main results one might hope for and give examples showing some of these are not possible. 
	
	\subsection{Normalizing the moduli space}\label{sec:normalization} It is natural to ask if Theorem \ref{theorem:reduction:are:iso} holds without taking the normalization of the moduli space $\sM_t$. However, the following example shows that the morphism $\beta_{t_i}$ is not injective in general and thus not an isomorphism. In particular, reduction morphisms $\rho_{\bfb, \bfa}$ can only be well-defined on the normalization of the moduli space. 
	
	We recall the following construction due to Hassett (see \cite[Example 2.41 and 1.42]{kollarmodulibook}). Consider the cone $S \subseteq \mathbb{P}^5$ over the
	degree four rational curve in $\mathbb{P}^4$. This is a surface with an $\mathbb{A}^2/\frac{1}{4}(1,1)$
	singularity on the vertex of the cone, and it can be obtained as a flat degeneration of both $\mathbb{P}^2$ (see \cite[Example 1.42]{kollarmodulibook}) and $\mathbb{P}^1 \times \mathbb{P}^1$. In particular, there are two DVRs, which we will denote by $R_1$ and $R_2$, and two projective families $f_i:\calX_i\to \Spec(R_i)$, so that the special fiber of $f_i$ is isomorphic to $S$, and the generic fiber of $f_1$ is
	isomorphic to $\mathbb{P}^1 \times \mathbb{P}^1$
	whereas the one of $f_2$ is isomorphic to $\mathbb{P}^2$. 
	
	Moreover, there are families of divisors $\calD_i \subseteq \calX_i$
	that can be described as follows. First, fix a natural number $r$ and let $D_0 \subset S$ be the union of $2r$ generic lines through the cone point. Now for $i=1$, the divisor $\calD_1$
	is the union of $r$ lines of one ruling on the generic fiber with divisorial limit $D_0$. Note that in this case $D_0$ is not the flat limit of the generic fiber $(\calD_1)_{\eta_1}$ but merely the divisorial component of the flat limit. Similarly, the divisor $\calD_2$ is $r$ general lines on the generic fiber $\mathbb{P}^2$ with divisorial limit $D_0$. In this case, it turns out that $D_0$ is actually the flat limit $(\calD_2)_{\eta_2}$. The pairs $(\calX_i,\frac{1}{r}\calD_i) \to \Spec(R_i)$ are projective locally stable families with special fiber $(S, \frac{1}{r}D_0)$.
	
	In particular, we can pick an $f_i$-ample hyperplane section $\calH_i\subseteq \calX_i$ as follows. First, choose an $f_i$-ample line bundle $\calL_i$ satisfying $R^1(f_i)_*(\calL_i^{\otimes m})=0$ for every $m>0$.
	Then any section of $H^0((\calL_i)_{|S})$ extends to a section of $H^0(\calL_i)$. This follows
	from the following exact sequence, where $t_i$ is the uniformizer of $\Spec(R_i)$:
	$$0 \to \calL_i \xrightarrow{\cdot t_i} \calL_i \to (\calL_i)_{|S} \to 0.$$
	In \cite[Example 2.41]{kollarmodulibook} it is shown that the divisor $K_{\calX_i}+\frac{1}{r}\calD_i$ is $\mathbb{Q}$-Cartier and anti-ample, so if we choose
	$\calL_i:=\calO_{\calX_i}(-n(K_{\calX_i}+\frac{1}{r}\calD_i))$ for $n$ divisible enough, then $R^1(f_i)_*(\calL_i^{\otimes m})=0$ for $m>0$, and $(\calL_1)_{|S}=(\calL_2)_{|S}$. In particular, we can choose two generic hyperplane sections $\calH_i \subseteq \calX_i$ of an appropriate multiple of $\calL_i$ so that
	$(\calH_1)_{|S}=(\calH_2)_{|S}$ and the divisor $(\calH_1)_{|S}$ avoids the singular locus of $S$ and intersects $D_0$ transversally. In particular, the pair $(S,\frac{1}{r}D_0+ \frac{1}{2}(\calH_i)_{|S})$ is lc, so by inversion of adjunction the morphisms $(\calX_i,\frac{1}{r}\calD_i+\frac{1}{2}\calH_i)\to \Spec(R_i)$ are stable. 
	
	This produces two stable families $(\calX_i,\frac{1}{r}\calD_i+\frac{1}{2}\calH_i) \to  \Spec(R_i)$ such that, if $\eta_i$ is the generic point of $\Spec(R_i)$, then:
	\begin{itemize}
		\item the generic fiber is klt;
		\item the special fibers are the same, but
		\item $K_{(\calX_1)_{\eta_1}}^2 \neq K_{(\calX_2)_{\eta_2}}^2$.
	\end{itemize}
	In particular, let $\mathbf{a}:=(\frac{1}{r},\frac{1}{2})$. Then for every $\varepsilon >0$ the pairs
	$$(\calX_i,(1-\varepsilon)(\frac{1}{r}\calD_i+\frac{1}{2}\calH_i))_{\eta_i}$$ have \emph{different} volumes.
	Therefore their stable limits along $\Spec(R_i)$ are two different points in $\sM_{\mathbf{a}-\varepsilon}$, but they have the same $\bfa$-weighted stable limit and thus map to the same point in $\sM_{\mathbf{a}}$. Therefore, the morphism $\beta_{\mathbf{a}}$ is not injective.
	
	\subsection{Multiple ways to reduce weights}\label{section:line} 
	
	In Theorem \ref{theorem:flip-like:morphisms}, we construct wall-crossing morphisms for the $(\bfa \to \bfb)$-walls along the line segment connecting the two weights vectors. However, the wall-and-chamber structure ultimately is a result of \cite[Theorem E]{BCHM10} which gives a decomposition of a polytope of weight vectors rather than just a line segment. Thus, it is natural to ask how the wall-crossing morphisms behave over the whole polytope. 
	
	In particular, given weight vectors $\bfa \ge \bfb \ge \bfc$ we have reduction morphisms $\rho_{\bfb, \bfa}$, $\rho_{\bfc, \bfb}$ and $\rho_{\bfc, \bfa}$ defined as a composition of flip-like morphisms and their inverses for the straight line segment $(\bfa \to \bfb)$, $(\bfb \to \bfc)$ and $(\bfa \to \bfc)$ respectively. Do these reduction morphisms commute? That is, do we have $\rho_{\bfc, \bfb} \circ \rho_{\bfb, \bfa} = \rho_{\bfc, \bfa}$ in general (see Remark \ref{rem:composition})? Recall that the construction of the flip-like morphisms proceeds by running a minimal model program with scaling as we reduce the coefficients along the corresponding line segment. The example below illustrates that these mmp with scaling do not commute in general and therefore the $\rho$ do not necessarily commute.

	We refer the reader to \cite{Mir89} for the background about on elliptic fibrations (see also \cite{lcmodelelliptic}).
	Consider a Weierstrass elliptic fibration $f:X \to \mathbb{P}^1$ with section $S$ and assume that the fundamental line bundle $\calL$ on $\mathbb{P}^1$ has degree 3. Consider five generic fibers $F_1,...,F_5$ and let $F:=\sum F_i$. Then we
	have that $K_X+(1-\varepsilon)S+ \frac{1}{2}F$ is ample, and the pair $(X,S+\frac{1}{2}F)$ is stable.
	For a suitable $\varepsilon$ small enough, the pair $(X,dS+\frac{d}{2}F)$ is also stable and klt, with $d:=1-\varepsilon$. 
	
	Recall now that
	\begin{itemize}
		\item $K_X=f^*(\omega_{\mathbb{P}^1} \otimes \calL)=f^*(\calO_{\mathbb{P}^1}(1))$, therefore
		\item $-2=(K_X+S).S=K_X.S+S^2=1+S^2$ so $S^2=-3$.
	\end{itemize}
	We aim at reducing the weights on $S$ and on $F$. We first reduce the weight on $S$ from $d$ to $\varepsilon$, for $\varepsilon > 0$ small enough. It is easy to check that the pair $(X,(td+(1-t)\varepsilon)S+\frac{d}{2}F)$ is stable for every $t \in [0,1]$. Now we can reduce the weight on $F$ from $\frac{d}{2}$ to $\frac{1}{5}$. Again, if $\varepsilon$ is small enough, it is easy to check that the pair $(X,\varepsilon S+(t\frac{d}{2}+(1-t)\frac{1}{5})F)$ is stable for every $t \in [0,1]$.

	On the other hand, we can first reduce the weights on $F$ first, and then on $S$. If we reduce the weights on $F$ from $\frac{d}{2}$ to $\frac{2}{5}-\varepsilon$, then $(K_X+dS+(\frac{2}{5}-\varepsilon)F).S<0$. In particular, the section $S$ must be contracted in the stable model. This gives a contraction morphism $g:X \to Y$, and a pseudoelliptic pair $(Y,g_*((\frac{2}{5}-\varepsilon)F))$. We can now keep reducing the weights from $\frac{2}{5}-\varepsilon$ to $\frac{1}{5}$ which produces a stable surface $(Z,D)$ with a contraction morphism $X \to Z$ which factors through $g$. In particular, $X$ and $Z$ are not isomorphic despite being the result of starting with the same $(1-\varepsilon, (1-\varepsilon)/2)$-weighted stable pair and reducing to coefficients $(\varepsilon, \frac{1}{5})$.
	
	One can produce a positive dimensional family of varieties of the above type by considering a Weierstrass fibration defined over the field $\mathbb{C}(t)$. This gives a morphism $\varphi : \Spec(\mathbb{C}(t)) \to \sK_{\mathbf{a}}$ whose closure of its image, assuming that $\varphi$ is non-isotrivial, will be a higher dimensional family of elliptic surfaces with generic fiber as in the example. In this case, the objects parametrized by the interior of the moduli spaces $\sM_t$ in Theorem \ref{theorem:flip-like:morphisms} depends on the chosen path from $\bfa \to \bfc$. 
	
	This shows that the moduli spaces $\sM_t$ and the flip-like morphisms depend \emph{a priori} on the choice of path. However, if we assume that we have a family such that the generic fiber has the same stable model for all coefficient vectors, then we can avoid this issue. More generally, suppose that there exists a polytope $P$ of admissible weight vectors and moduli spaces $\sM_{\bfv}$ of $\bfv$-weighted stable models for each $\bfv \in P$ such that 
	\begin{itemize} 
		\item there are dense open substacks $\calU_{\bfv} \subset \sM_{\bfv}$ with reduction morphisms $r_{\bfb, \bfa} : \calU_{\bfa} \to \calU_{\bfb}$ for $\bfb \le \bfa$, and
		\item for every $\bfc \le \bfb \le \bfa$ in $P$, we have $r_{\bfc, \bfb} \circ r_{\bfb, \bfa} = r_{\bfc, \bfa}$.
	\end{itemize}
	Then since the moduli spaces are separated, we must have that $\rho_{\bfc, \bfb} \circ \rho_{\bfb, \bfa} = \rho_{\bfc, \bfa}$ (see \cite[Lemma 7.2]{DH18}). This applies for example in the hypothesis of Theorem \ref{theorem:summary}.

	More generally, we have proved the following. 
	\begin{theorem}\label{thm:last:thm}
		Let $\pi : (\calX,\calD_1,...,\calD_n) \to \sM_{\bfa}$ be the universal family of pairs over $\sM_{\bfa}$ and assume that for every $\mathbf{v} \le \bfa$ in an admissible polytope of weight vectors $P$ (Definition \ref{def:admissible}), the generic point of $\pi : (\calX,\mathbf{v}\calD) \to \sM_{\bfa}$ is a stable family. Then the moduli spaces $\sM_{\mathbf{v}}$ and morphisms $\rho_{\mathbf{b},\mathbf{c}} \circ \rho_{\mathbf{a},\mathbf{b}}=\rho_{\mathbf{a},\mathbf{c}}$ are well-defined for every pair $\mathbf{c}\le \mathbf{b}$ in $P$.
	\end{theorem}
	
	\subsection{Reduction morphisms are not birational onto its image}\label{sec:notbirationa}

	We give an example where the reduction morphisms are not birational if we do not assume that the generic fiber is stable for each $t \in [0,1]$. As we noted in Theorem \ref{thm:reduction:morphism:birational:for:ZMT}, the morphisms $\beta_{t_i}$ are always birational. The following example shows that the morphisms $\alpha_{t_i}$ are \emph{not} always birational.

	Consider pairs $(\bP^2, C_0 + L_0)$ where
	\begin{itemize}
		\item $C_0$ is a generic curve of degree $d \gg 0$, and
		\item $L_0$ is a generic line. 
	\end{itemize}
	By the genericity assumption, $C_0$ and $L_0$ meet transversely at $d$ points. Let $p \in C_0 \cap L_0$ and let $\mu : X \to \bP^2$ be the blowup of $\bP^2$ at $p$ with exceptional divisor $E$. Let $C$ and $L$ be the strict transforms of $C_0$ and $L_0$ respectively. Consider the pair $(X, aC + aL)$ for some $a \in (0,1)$. We can compute
	\begin{equation}\label{eqn:disc}
		K_X + aC + aL = f^*(K_{\bP^2} + aC_0 + aL_0) + (1-2a)E.  
	\end{equation}
	For $a = \frac{1 + \varepsilon}{2}$ with $0 < \varepsilon \ll 1$, the divisor $K_X + aC + aL$ is ample so $\left(X, \frac{1 + \varepsilon}{2}(C + L)\right)$ is stable. On the other hand, at $a = \frac{1}{2}$ the canonical model contracts the exceptional curve $E$ to obtain $\left(\bP^2, \frac{1}{2}(C_0 + L_0)\right)$ as the log canonical model. 
	
	Let $\calU$ be the moduli space of such pairs $(X, aC + aL)$. Explicitly, we can construct $\calU$ as a $d$-fold cover of
	$$
	\left[U/\mathrm{PGL}_3\right] \subset \left[(H_1 \times H_d)/\mathrm{PGL}_3\right]
	$$
	where $H_i$ is the Hilbert scheme of degree $i$ curves and the $d$-fold cover corresponds to a choice of point in $C_0 \cap L_0$ to blow up. Then $\calU$ is smooth and supports a universal family
	$$
	\left(\calX, a\calC + b\calL\right)
	$$
	of such pairs $(X, aC + aL)$. This is an open subset of the moduli of stable log pairs for $a = \frac{1 + \varepsilon}{2}$ but at coefficient $a = \frac{1}{2}$, the reduction morphism $\rho_{\frac{1}{2}, \frac{1 + \varepsilon}{2}}$ can be identified with the $d$-fold cover
	$$
	\calU \to \left[U/\mathrm{PGL}_3\right]
	$$
	as $\rho$ blows $X$ down to $\bP^2$ and thus forgets about the choice of $p \in C_0\cap L_0$.

	\subsection{Reduction morphisms are not dominant on irreducible components}\label{sec:JK}
	It is natural to ask if the image of the normalization of an irreducible component of $\sK_\bfa$ under a reduction map $\rho_{\bfb, \bfa}$ is the normalization of an irreducible component of $\sK_{\bfb}$. More generally, one can ask if the image of the klt locus under $\rho$ is open. While this is true in dimension one, it fails in higher dimensions. The following example of this phenomena was pointed out to us by J\'anos Koll\'ar, answering questions that appeared in an earlier version of this manuscript. We thank him for allowing us to include it here.
	
	\begin{example}[Koll\'ar]\label{ex:JK}Consider $(Q_0, uL_0 + vL^\prime_0)$, where $Q_0$ is a quadric cone, and both $L$ and $L^\prime$ are lines. If $u=v$ then this pair deforms to $(Q_t, uL_t + vL^\prime_t)$, where $Q_t$ is a smooth quadric cone, and $L_t$ and $L^\prime_t$ are now lines in different families, but if $u \neq v$ then there is no such deformation. Thus, the image of the irreducible component of the moduli space parametrizing such pairs at coefficient $(u,v)$ does not dominate an irreducible component of the moduli space for coefficients $(u,u)$ where $u < v$. 
	\end{example}

	\subsection{Further questions}
	
	We end with a few natural open questions. In the examples of Sections \ref{section:line} and \ref{sec:notbirationa}, what seems to go wrong is that the stable model contracts marked divisors of our pairs. Thus, it is natural to ask if this is the only thing that can go wrong. 
	
	\begin{question} In the setting of Theorem \ref{thm:mainintro}, if we further assume that the stable models of $$(X, \sum a_i D_i) \to B$$ are isomorphic in codimension one for each admissible weight $\bfa$, do the stronger conclusions of Theorem \ref{theorem:summary} hold?
	\end{question}
	
	Finally, we make use throughout of the klt assumption for the generic fiber of our universal family in order to apply the results of \cite{BCHM10} among other things. It is natural to ask the following. 
	
	\begin{question} Do the wall-crossing results of this paper hold if we only assume that the generic fiber of the universal family over the moduli space is log canonical rather than klt?
	\end{question} 
	
	\begin{remark}\label{rmk:abundance}
		It is also natural to ask if the main results of this paper hold when the generic fiber is log canonical while assuming the full minimal model program and abundance. We are using the klt assumption at least in Proposition \ref{prop:MMP:with:scaling:Jakub} and
		Theorem \ref{theorem:flip-like:morphisms} when we used Kawamata--Viehweg vanishing, and in Theorem \ref{thm:morphismextends} where we used \cite{BCHM10}. We do expect our results to hold in the case where the generic fiber of our moduli spaces admit a good minimal model, and will leave it for future exploration.
	\end{remark}

	\bibliographystyle{alpha}
	\bibliography{wallcrossing.bib}

\end{document}